\documentclass[10pt]{article}
\usepackage{latexsym,amssymb, amsmath}
\usepackage{amsthm,upref,amscd}
\usepackage{color}
\usepackage[T1]{fontenc}

\begin{document}
\baselineskip=14pt

\numberwithin{equation}{section}

\newtheorem{thm}{Theorem}[section]
\newtheorem{lem}[thm]{Lemma}
\newtheorem{cor}[thm]{Corollary}
\newtheorem{Prop}[thm]{Proposition}
\newtheorem{Def}[thm]{Definition}
\newtheorem{Rem}[thm]{Remark}
\newtheorem{Ex}[thm]{Example}

\newcommand{\A}{\mathbb{A}}
\newcommand{\B}{\mathbb{B}}
\newcommand{\C}{\mathbb{C}}
\newcommand{\D}{\mathbb{D}}
\newcommand{\E}{\mathbb{E}}
\newcommand{\F}{\mathbb{F}}
\newcommand{\G}{\mathbb{G}}
\newcommand{\I}{\mathbb{I}}
\newcommand{\J}{\mathbb{J}}
\newcommand{\K}{\mathbb{K}}
\newcommand{\M}{\mathbb{M}}
\newcommand{\N}{\mathbb{N}}
\newcommand{\Q}{\mathbb{Q}}
\newcommand{\R}{\mathbb{R}}
\newcommand{\T}{\mathbb{T}}
\newcommand{\U}{\mathbb{U}}
\newcommand{\V}{\mathbb{V}}
\newcommand{\W}{\mathbb{W}}
\newcommand{\X}{\mathbb{X}}
\newcommand{\Y}{\mathbb{Y}}
\newcommand{\Z}{\mathbb{Z}}
\newcommand\ca{\mathcal{A}}
\newcommand\cb{\mathcal{B}}
\newcommand\cc{\mathcal{C}}
\newcommand\cd{\mathcal{D}}
\newcommand\ce{\mathcal{E}}
\newcommand\cf{\mathcal{F}}
\newcommand\cg{\mathcal{G}}
\newcommand\ch{\mathcal{H}}
\newcommand\ci{\mathcal{I}}
\newcommand\cj{\mathcal{J}}
\newcommand\ck{\mathcal{K}}
\newcommand\cl{\mathcal{L}}
\newcommand\cm{\mathcal{M}}
\newcommand\cn{\mathcal{N}}
\newcommand\co{\mathcal{O}}
\newcommand\cp{\mathcal{P}}
\newcommand\cq{\mathcal{Q}}
\newcommand\rr{\mathcal{R}}
\newcommand\cs{\mathcal{S}}
\newcommand\ct{\mathcal{T}}
\newcommand\cu{\mathcal{U}}
\newcommand\cv{\mathcal{V}}
\newcommand\cw{\mathcal{W}}
\newcommand\cx{\mathcal{X}}
\newcommand\ocd{\overline{\cd}}

\def\c{\centerline}
\def\ov{\overline}
\def\emp {\emptyset}
\def\pa {\partial}
\def\bl{\setminus}
\def\op{\oplus}
\def\sbt{\subset}
\def\un{\underline}
\def\al {\alpha}
\def\bt {\beta}
\def\de {\delta}
\def\Ga {\Gamma}
\def\ga {\gamma}
\def\lm {\lambda}
\def\Lam {\Lambda}
\def\om {\omega}
\def\Om {\Omega}
\def\sa {\sigma}
\def\vr {\varepsilon}
\def\va {\varphi}

\title{\bf \LARGE On nonlocal Choquard equations with Hardy-Littlewood-Sobolev critical exponents \thanks{Partially supported by NSFC (11571317, 11101374, 11271331) and ZJNSF(LY15A010010)}\vspace{5mm}}

\author{ Fashun Gao and Minbo Yang\thanks{M. Yang is the corresponding author: mbyang@zjnu.edu.cn}
\vspace{2mm}\\
{\small Department of Mathematics, Zhejiang Normal University} \\ {\small  Jinhua, Zhejiang, 321004, P. R. China}}

\date{}
\maketitle

\begin{abstract}
We consider the following  nonlinear Choquard equation with Dirichlet boundary condition
$$-\Delta u
=\left(\int_{\Omega}\frac{|u|^{2_{\mu}^{\ast}}}{|x-y|^{\mu}}dy\right)|u|^{2_{\mu}^{\ast}-2}u+\lambda f(u)\hspace{4.14mm}\mbox{in}\hspace{1.14mm} \Omega,
$$
where $\Omega$ is a smooth bounded domain of $\mathbb{R}^N$, $\lambda>0$, $N\geq3$,  $0<\mu<N$ and $2_{\mu}^{\ast}$ is the critical exponent in the sense of the Hardy-Littlewood-Sobolev inequality. Under suitable assumptions on different types of nonlinearities $f(u)$,  we are able to prove some existence and multiplicity results for the equation by variational methods.
 \vspace{0.3cm}


\vspace{0.3cm}

\end{abstract}

\section{\large Introduction and main results}

In a pioneering paper \cite{BN} , Brezis and Nirenberg studied the problems of the type
\begin{equation}\label{local.S1}
\left\{\begin{array}{l}
\displaystyle-\Delta u=|u|^{2^{\ast}-2}u+\lambda u^{q}\ \ \mbox{in}\ \ \Omega,\\
\\
\displaystyle u=0 \hspace{25.14mm}\ \ \ \mbox{on}\ \ \partial\Omega,
\end{array}
\right.
\end{equation}
where $2^{\ast}$  is the critical exponent  for the embedding of $H_{0}^{1}(\Omega)$ to $L^p(\Omega)$.  By developing some skillful techniques in estimating the Minimax level, the authors were able to prove the existence of  nontrivial solutions for the equation under linear or subcritical superlinear perturbation. To complement the results in \cite{BN}, the sublinear case was studied by Ambrosetti, Brezis and Cerami in \cite{ABC}. There the authors investigated the combined effects of the critical term and the sublinear term on the existence of solutions and established the existence and multiplicity results depending on the parameter $\lambda$. Ever since then, the elliptic equations with critical  exponents have been a hot subject. For example, \cite{AM, BW, CFP, CFS, CSS, GaPa, GP, Wi} studied the critical problems on bounded or unbounded domains, \cite{DH,GYu} investigated the quasilinear critical growth problems driven by  $p$-Laplacian, \cite{CS} studied the regularity of stable solutions of elliptic problems on Riemannian manifold, \cite{ CP, DMPV, FM, GYu, J, RSW, SV} considered the Hardy-Sobolev type inequalities with remaining terms and the Hardy-Sobolev singular critical exponent problems, while \cite{BCDS1, BCSS, Ta} investigated the critical problems driven by the  fractional power of the Laplacian $(-\Delta)^{s}$ ($0<s<1$).

In the present paper we are going to study the existence and multiplicity results for the following {\it critical nonlocal equation}:
\begin{equation}\label{CCE}
\left\{\begin{array}{l}
\displaystyle-\Delta u
=\Big(\int_{\Omega}\frac{|u|^{2_{\mu}^{\ast}}}{|x-y|^{\mu}}dy\Big)|u|^{2_{\mu}^{\ast}-2}u+\lambda f(u)\hspace{4.14mm}\mbox{in}\hspace{1.14mm} \Omega,\\
\displaystyle u\in H_{0}^{1}(\Omega),
\end{array}
\right.
\end{equation}
where $\Omega$ is a smooth bounded domain of $\mathbb{R}^N$, $\lambda>0$, $N\geq3$, $0<\mu<N$, $2_{\mu}^{\ast}=(2N-\mu)/(N-2)$ and $f(u)$ is a nonlinearity satisfying certain assumptions. This type of nonlocal elliptic equation is closely related to the nonlinear Choquard equation
\begin{equation}\label{Nonlocal.S1}
 -\Delta u +V(x)u =\Big(\frac{1}{|x|^{\mu}}\ast|u|^{p}\Big)|u|^{p-2}u  \quad \mbox{in} \quad \R^3,
\end{equation}
where $
\frac{2N-\mu}{N}\leq p\leq\frac{2N-\mu}{N-2}
$. For $p=2$ and $\mu=1$, the problem goes
back to the description of the quantum theory of a polaron at rest by S. Pekar in 1954 \cite{P1}
and the modeling of an electron trapped
in its own hole in 1976 in the work of P. Choquard, as a certain approximation to Hartree-Fock theory of one-component
plasma \cite{L1}. In some particular cases, this equation is also known as the Schr\"{o}dinger-Newton equation, which was introduced by Penrose in his discussion on the selfgravitational collapse of a quantum mechanical wave function  \cite{Pe}. For $p=\frac{2N-1}{N-2}$ and $\mu=1$, by using the Green function, it is obvious that \eqref{Nonlocal.S1} can been regarded as a generalized version of Schr\"{o}dinger-Newton
$$
\left\{\begin{array}{l}
\displaystyle-\Delta u+V(x)u =u^{\frac{N+1}{N-2}}\phi\ \ \ \ \mbox{in}\ \ \mathbb{R}^N,\\
\\
\displaystyle-\Delta \phi=u^{\frac{2N-1}{N-2}}\ \ \ \ \ \ \ \ \ \ \  \ \ \mbox{in}\ \ \mathbb{R}^N.\\
\end{array}
\right.
$$
And thus equation \eqref{CCE} can been viewed as a generalized Schr\"{o}dinger-Newton restricted on bounded domain with Dirichlet boundary condition.

The starting point of the variational approach to the problem \eqref{CCE} is the following  well-known Hardy-Littlewood-Sobolev inequality, see \cite{LL}, which leads to a new type of critical problem with nonlocal nonlinearities driven by Riesz potential.
\begin{Prop}\label{HLS}
 [Hardy-Littlewood-Sobolev inequality]. Let $t,r>1$ and $0<\mu<N$ with $1/t+\mu/N+1/r=2$, $f\in L^{t}(\mathbb{R}^N)$ and $h\in L^{r}(\mathbb{R}^N)$. There exists a sharp constant $C(t,N,\mu,r)$, independent of $f,h$, such that
\begin{equation}\label{HLS1}
\int_{\mathbb{R}^{N}}\int_{\mathbb{R}^{N}}\frac{f(x)h(y)}{|x-y|^{\mu}}dxdy\leq C(t,N,\mu,r)|f|_{t}|h|_{r},
\end{equation}
where $|\cdot|_{q}$ for the $L^{q}(\mathbb{R}^{N})$-norm for $q\in[1,\infty]$. If $t=r=2N/(2N-\mu)$, then
$$
 C(t,N,\mu,r)=C(N,\mu)=\pi^{\frac{\mu}{2}}\frac{\Gamma(\frac{N}{2}-\frac{\mu}{2})}{\Gamma(N-\frac{\mu}{2})}\left\{\frac{\Gamma(\frac{N}{2})}{\Gamma(N)}\right\}^{-1+\frac{\mu}{N}}.
$$
In this case there is equality in \eqref{HLS1} if and only if $f\equiv(const.)h$ and
$$
h(x)=A(\gamma^{2}+|x-a|^{2})^{-(2N-\mu)/2}
$$
for some $A\in \mathbb{C}$, $0\neq\gamma\in\mathbb{R}$ and $a\in \mathbb{R}^{N}$.
\end{Prop}

The Hardy-Littlewood-Sobolev inequality plays an important role in studying nonlocal problems and we'd like to mention that other nonlocal version inequalities are considered in some recent literature, for example, the authors in \cite{DV} studied the Hardy-Littlewood inequalities in fractional weighted Sobolev spaces.

Notice that, by the Hardy-Littlewood-Sobolev inequality, the integral
$$
\int_{\mathbb{R}^{N}}\int_{\mathbb{R}^{N}}\frac{|u(x)|^{q}|u(y)|^{q}}{|x-y|^{\mu}}dxdy
$$
is well defined if
$$
\frac{2N-\mu}{N}\leq q\leq\frac{2N-\mu}{N-2}.
$$
Here, $\frac{2N-\mu}{N}$ is called the lower critical exponent and $2_{\mu}^{\ast}=\frac{2N-\mu}{N-2}$ is the upper critical exponent due to the Hardy-Littlewood-Sobolev inequality. In this sense we can call the problem \eqref{CCE} a critical nonlocal elliptic equation.  In a recent paper \cite{MS4} by Moroz and Van
Schaftingen, the authors considered  the nonlinear Choquard equation \eqref{Nonlocal.S1} in $\R^N$ with lower critical exponent $\frac{2N-\mu}{N}$ if  the potential $1-V$ should not decay to zero at infinity faster then the inverse square of $|x|$. However, nothing is known about the upper critical exponent case and how the behavior of the potential will affect the existence results.

The existence and qualitative properties of solutions of Choquard type  equations \eqref{Nonlocal.S1} have been widely studied in the last decades. In \cite{L1}, Lieb proved the existence and uniqueness, up to translations,
of the ground state. Later, in \cite{Ls}, Lions showed the existence of
a sequence of radially symmetric solutions. In \cite{CCS1, ML,  MS1} the authors showed the regularity, positivity
and radial symmetry of the ground states and
derived decay property at infinity as well. Nowadays the existence and asymptotic behavior of the solutions of the Choquard equations under different assumptions of the potentials and the nonlinearities attract a lot of attention and interest.  We refer the readers to \cite{AC, BJS} for the strongly indefinite case with sign-changing periodic potential $V$, where the existence of ground state solutions and infinitely many geometrically distinct weak solutions were obtained by critical point theorem; \cite{MS2} for  the existence of ground states under the assumptions of Berestycki-Lions type; \cite{ANY} for the existence of multibump shaped solution for the equation with deepening potential well; \cite{CS, GS} for the existence of sign-changing solutions; \cite{AY1, AY2, MS3, WW} for the existence and concentration behavior of the semiclassical solutions.

To the best knowledge of the authors, there are not so many papers considering the Choquard equation with critical growth, for problem set on $\R^2$  the second author and his collaborators in \cite{ACTY} considered the case of critical growth in the sense of Trudinger-Moser inequality and studied the existence and concentration of the ground states. For the problem set on a bounded domain of $\R^N$, $N\geq 3$, the authors of the present paper considered  in  \cite{GY}  the critical Choquard equation  \eqref{CCE}  with upper critical exponent $2_{\mu}^{\ast}=\frac{2N-\mu}{N-2}$  under  a linear perturbation term  and established the existence results corresponding to the well-known results in \cite{BN}.

The aim of the present paper is to continue to study the existence and multiplicity of the critical Choquard equation  \eqref{CCE} with upper critical exponent $2_{\mu}^{\ast}=\frac{2N-\mu}{N-2}$ on a bounded domain of $\R^N$, $N\geq 3$. We are interested in the problem that how  the subcritical superlinear or the sublinear perturbation term will affect the existence and multiplicity of  the critical Choquard equation  \eqref{CCE}. In \cite{GY}, $S_{H,L}$ denotes the best constant defined by
\begin{equation}\label{S1}
S_{H,L}:=\displaystyle\inf\limits_{u\in D^{1,2}(\mathbb{R}^N)\backslash\{{0}\}}\ \ \frac{\displaystyle\int_{\mathbb{R}^N}|\nabla u|^{2}dx}{(\displaystyle\int_{\mathbb{R}^N}\int_{\mathbb{R}^N}
\frac{|u(x)|^{2_{\mu}^{\ast}}|u(y)|^{2_{\mu}^{\ast}}}{|x-y|^{\mu}}dxdy)^{\frac{N-2}{2N-\mu}}}.
\end{equation}

\begin{Prop}\label{ExFu} (See \cite{GY}.)
The constant $S_{H,L}$ defined in \eqref{S1} is achieved if and only if $$u=C\Big(\frac{b}{b^{2}+|x-a|^{2}}\Big)^{\frac{N-2}{2}} ,$$ where $C>0$ is a fixed constant, $a\in \mathbb{R}^{N}$ and $b\in(0,\infty)$ are parameters. What is more,
$$
S_{H,L}=\frac{S}{C(N,\mu)^{\frac{N-2}{2N-\mu}}},
$$
where $S$ is the best Sobolev constant.
\end{Prop}
Let $U(x):=\frac{[N(N-2)]^{\frac{N-2}{4}}}{(1+|x|^{2})^{\frac{N-2}{2}}}$ be a minimizer for $S$, see \cite{Wi} for example, then
\begin{equation}\label{REL}
\aligned
\tilde{U}(x)=S^{\frac{(N-\mu)(2-N)}{4(N-\mu+2)}}C(N,\mu)^{\frac{2-N}{2(N-\mu+2)}}\frac{[N(N-2)]^{\frac{N-2}{4}}}{(1+|x|^{2})^{\frac{N-2}{2}}}
\endaligned
\end{equation}
is the unique  minimizer for $S_{H,L}$ that satisfies
$$
-\Delta u=\Big(\int_{\R^N}\frac{|u|^{2_{\mu}^{\ast}}}{|x-y|^{\mu}}dy\Big)|u|^{2_{\mu}^{\ast}-2}u\ \ \   \hbox{in}\ \ \ \R^N
$$
with
$$
\int_{\mathbb{R}^N}|\nabla \tilde{U}|^{2}dx=\int_{\mathbb{R}^N}\int_{\mathbb{R}^N}\frac{|\tilde{U}(x)|^{2_{\mu}^{\ast}}|\tilde{U}(y)|^{2_{\mu}^{\ast}}}{|x-y|^{\mu}}dxdy=S_{H,L}^{\frac{2N-\mu}{N-\mu+2}}.
$$
Moreover, for every open subset $\Omega$ of $\mathbb{R}^N$,
\begin{equation}
S_{H,L}(\Omega):=\displaystyle\inf\limits_{u\in D_{0}^{1,2}(\Omega)\backslash\{{0}\}}\ \ \frac{\displaystyle\int_{\Omega}|\nabla u|^{2}dx}{\Big(\displaystyle\int_{\Omega}\int_{\Omega}\frac{|u(x)|^{2_{\mu}^{\ast}}|u(y)|^{2_{\mu}^{\ast}}}{|x-y|^{\mu}}dxdy\Big)^{\frac{N-2}{2N-\mu}}}=S_{H,L},
\end{equation}
and $S_{H,L}(\Omega)$ is never achieved except  $\Omega=\R^N$, see \cite{GY}.

Through this paper we will assume that $\Omega$ is a smooth bounded domain of $\mathbb{R}^N$, the main results of this paper are the following Theorems.

For the critical Choquard equation \eqref{CCE} with a subcritical local term  $f=u^{q}$,  $1<q<2^{\ast}-1$,
\begin{equation}\label{CCE1}
\left\{\begin{array}{l}
\displaystyle-\Delta u
=\Big(\int_{\Omega}\frac{|u|^{2_{\mu}^{\ast}}}{|x-y|^{\mu}}dy\Big)|u|^{2_{\mu}^{\ast}-2}u+\lambda u^q\hspace{4.14mm}\mbox{in}\hspace{1.14mm} \Omega,\\
\displaystyle u\in H_{0}^{1}(\Omega),
\end{array}
\right.
\end{equation}
we have the following existence result.
\begin{thm}\label{EXS}
Assume that $1<q<2^{\ast}-1$, $N\geq3$ and $0<\mu<N$. Then, problem \eqref{CCE1} has at least one nontrivial solution provided that either
\begin{itemize}
\item[(1).] $N>\max\{\min\{\frac{2(q+3)}{q+1},2+\frac{\mu}{q+1}\},\frac{2(q+1)}{q}\}$ and $\lambda>0$, or
\item[(2).] $N\leq\max\{\min\{\frac{2(q+3)}{q+1},2+\frac{\mu}{q+1}\},\frac{2(q+1)}{q}\}$ and $\lambda$ is sufficiently large.
\end{itemize}
\end{thm}

We are also interested in the critical Choquard equation \eqref{CCE} with a subcritical nonlocal term, that is to consider the following equation,
\begin{equation}\label{CCE2}
\left\{\begin{array}{l}
\displaystyle-\Delta u
=\Big(\int_{\Omega}\frac{|u|^{2_{\mu}^{\ast}}}{|x-y|^{\mu}}dy\Big)|u|^{2_{\mu}^{\ast}-2}u+\lambda \Big(\int_{\Omega}\frac{|u|^{q}}{|x-y|^{\mu}}dy\Big)|u|^{q-2}u\hspace{4.14mm}\mbox{in}\hspace{1.14mm} \Omega,\\
\displaystyle u\in H_{0}^{1}(\Omega).
\end{array}
\right.
\end{equation}
For this case, we establish the following existence result.
\begin{thm}\label{EXS2}
Assume that $1< q<2_{\mu}^{\ast}$, $N\geq3$ and $0<\mu<N$. Then, problem \eqref{CCE2} has at least one nontrivial solution provided that either
\begin{itemize}
\item[(1).] $N>\frac{2(q+1)-\mu}{q-1}$ and $\lambda>0$, or
\item[(2).] $N\leq\frac{2(q+1)-\mu}{q-1}$ and $\lambda$ is sufficiently large.
\end{itemize}
\end{thm}

Analogously, we have the existence result for the Choquard equation with Sobolev critical exponent and subcritical nonlocal perturbation.
\begin{thm}\label{EXS3}
Assume that $1< q<2_{\mu}^{\ast}$, $N\geq3$ and $0<\mu<N$. Then, problem \begin{equation}
\left\{\begin{array}{l}
\displaystyle-\Delta u
=\lambda \Big(\int_{\Omega}\frac{|u|^{q}}{|x-y|^{\mu}}dy\Big)|u|^{q-2}u+ |u|^{2^{\ast}-2}u\hspace{4.14mm}\mbox{in}\hspace{1.14mm} \Omega,\\
\displaystyle u\in H_{0}^{1}(\Omega)
\end{array}
\right.
\end{equation}
has at least one nontrivial solution provided that either
\begin{itemize}
\item[(1).] $N>\frac{2(q+1)-\mu}{q-1}$ and $\lambda>0$, or
\item[(2).] $N\leq\frac{2(q+1)-\mu}{q-1}$ and $\lambda$ is sufficiently large.
\end{itemize}
\end{thm}

As observed by Ambrosetti, Brezis and Cerami in \cite{ABC}, the combination of a critical term and sublinear term influence the existence and multiplicity of the equation \eqref{local.S1} greatly. Here we can also establish a similar result for the critical Choquard equation under the perturbation of both sublinear and suplinear subcritical terms. Consider the following equation
\begin{equation}\label{SCCE2}
\left\{\begin{array}{l}
\displaystyle-\Delta u
=\Big(\int_{\Omega}\frac{|u|^{2_{\mu}^{\ast}}}{|x-y|^{\mu}}dy\Big)|u|^{2_{\mu}^{\ast}-2}u+u^{p}+\lambda u^{q}\hspace{4.14mm}\mbox{in}\hspace{1.14mm} \Omega,\\
\displaystyle u\in H_{0}^{1}(\Omega),
\end{array}
\right.
\end{equation}
we can draw the following conclusions.
\begin{thm}\label{EXS4}
Assume that $0<q<1$, $1<p<2^*-1$, $N\geq3$ and $0<\mu<N$. Then, there exists $0<\Lambda<\infty$ such that
\begin{itemize}
 		\item [(1).] problem \eqref{SCCE2} has no positive solution for $\lambda>\Lambda$;
\item[(2).] problem \eqref{SCCE2} has a minimal positive solution $u_{\lambda}$ for any $0<\lambda<\Lambda$ and the family of minimal solutions is increasing with respect to $\lambda$;
\item [(3).] problem \eqref{SCCE2} has at least two positive solutions if $0<\lambda<\Lambda$.
\end{itemize}
\end{thm}

From the statements above, we had completely studied how the different perturbation will affect the existence results. In fact, Theorem \ref{EXS} states the existence under a superlinear perturbation while Theorem \ref{EXS2} and Theorem \ref{EXS3} focuses on the nonlocal perturbation both for the critical problem in the sense of the Hardy-Littlewood-Sobolev inequality and the classical Sobolev inequality. Finally Theorem \ref{EXS4} says that how the appearance of sublinear perturbation will lead to the existence results.

Next we are going to investigate the existence of infinitely many solutions for the critical Choquard equation \eqref{CCE1} under the effect of only a sublinear local perturbation.

\begin{thm}\label{MEXS}
Assume that $0<q<1$, $N\geq3$ and $0<\mu<N$. Then, there exists $\lambda^{\ast}>0$ such that, for every $0<\lambda<\lambda^{\ast}$, problem \eqref{CCE1} has a sequence of solutions $\{u_{n}\}\subset H_{0}^{1}(\Omega)$ such that $J_{\lambda}(u_{n})\rightarrow0$, $n\rightarrow\infty$.
\end{thm}

We would like to point out that, for the convex and concave problem with subcritical growth, it is standard to apply the Fountain Theorem to obtain  an unbounded sequence of solutions. We state the multiplicity result here for the completeness of this paper.
\begin{thm}\label{MEXS2}
Assume that $0<q<1$, $N\geq3$, $0<\mu<N$ and $1\leq p<2_{\mu}^{\ast}$. Then, for every $\lambda>0$, the problem
\begin{equation}\label{SCCE}
\left\{\begin{array}{l}
\displaystyle-\Delta u
=\Big(\int_{\Omega}\frac{|u|^{p}}{|x-y|^{\mu}}dy\Big)|u|^{p-2}u+\lambda u^{q}\hspace{4.14mm}\mbox{in}\hspace{1.14mm} \Omega,\\
\displaystyle u\in H_{0}^{1}(\Omega),
\end{array}
\right.
\end{equation}
has an unbounded sequence of solutions $\{u_{n}\}\subset H_{0}^{1}(\Omega)$ such that $J_{\lambda}(u_{n})\rightarrow\infty$, $n\rightarrow\infty$.
\end{thm}

The main results will be proved by variational methods. For this, we introduce the energy functional associated to equation \eqref{CCE} by
$$
J_{\lambda}(u)=\frac{1}{2}\int_{\Omega}|\nabla u|^{2}dx-\frac{1}{2\cdot2_{\mu}^{\ast}}\int_{\Omega}
\int_{\Omega}\frac{|u(x)|^{2_{\mu}^{\ast}}|u(y)|^{2_{\mu}^{\ast}}}
{|x-y|^{\mu}}dxdy-\lambda\int_{\Omega}F(u)dx,
$$
 where $F(u)=\int_{0}^{u}f(t)dt$.
Then, under the assumptions in the theorems above, the Hardy-Littlewood-Sobolev inequality implies that $J_{\lambda}$ belongs to $C^{1}(H_{0}^{1}(\Omega),\R)$ with
$$
\langle J_{\lambda}^{'}(u),\varphi\rangle=\int_{\Omega}\nabla u\nabla\varphi dx-\int_{\Omega}\int_{\Omega}\frac{|u(x)|^{2_{\mu}^{\ast}}|u(y)|^{2_{\mu}^{\ast}-2}u(y)\varphi(y)}
{|x-y|^{\mu}}dxdy-\lambda\int_{\Omega}f(u)\varphi dx,\ \ \forall\varphi\in C_{0}^{\infty}(\Omega).
$$
And so $u$ is a weak solution of \eqref{CCE} if and only if $u$ is a critical point of functional $J_{\lambda}$.

We need to point out the main features of the present problem are two-fold: the first is the loss of compactness due to the appearance of the Hardy-Littlewood-Sobolev upper critical exponent which makes it difficult to verify the $(PS)$ condition. The second is the nonlocal nature of the critical Choquard equation where the convolution type nonlinearities, no longer locally defined, are totally determined by the behavior on the domain $\omega$.
This feature not only makes it difficult to verify the geometric conditions of the critical point theorems, but also causes a lot of trouble in applying the well-known Brezis-Nirenberg type arguments and showing the regularity of the solutions.

For the proof of the main results, we will show that the energy functional $J_{\lambda}$ satisfies a compactness property and has suitable geometrical features of some critical point theorems. Frequently, we will apply the Mountain Pass Theorem to study the critical Choquard equation with subcritical superlinear perturbation and obtain the existence of at least one solution for \eqref{CCE} for suitable values of $\lambda$, the space dimension $N$ and the order of the Riesz potential $\mu$. As in \cite{BN}, the key step is to use the extreme function of the best constant in Proposition \ref{ExFu} to estimate the Mountain Pass value. By showing that the Minimax value is below the level where the $(PS)$ condition holds, one can easily obtain the existence results.
The proof of Theorem \ref{EXS3} is similar to that of Theorem \ref{EXS} and the proof of Theorem \ref{EXS2} will only be sketched. For the critical Choquard equation with only sublinear perturbation, we will use the Dual Fountain Theorem established in \cite{BW} to prove that the problem \eqref{CCE} has a sequence of solutions $\{u_{n}\}\subset H_{0}^{1}(\Omega)$ such that $J_{\lambda}(u_{n})\rightarrow0$, $n\rightarrow\infty$ and we also write some lines for the subcritical nonlocal case for the completeness of the paper. For the critical Choquard equation with both sublinear and superlinear perturbation, we will follow the idea in \cite{ABC} to prove the existence of at least two positive solutions for an admissible small range of $\lambda$. Firstly, we can obtain a positive solution that is a local minimum for the functional $J_{\lambda}$ by sub- and supersolution techniques. Secondly, in order to find a second solution of \eqref{SCCE2}, we suppose that this local minimum is the only critical point of the functional $J_{\lambda}$ and then we prove a local Palais-Smale ($(PS)$ for short) condition for $c\in\R$ below a critical level related with $S_{H,L}$ defined in \eqref{S1}. Finally, we apply the Mountain Pass Theorem and its refined version in \cite{GP} to get the conclusion. In order to overcome the difficulties, we need to adjust the arguments in \cite{BN, ABC} to suit the new environment.\vspace{5mm}\\
\par
Throughout the paper,we will use the following notations: \\
\noindent $\bullet$ We denote positive constants by $C, C_{1}, C_{2}, C_{3}, \cdots$ and $\lambda>0$ is a real parameter.\\
\noindent $\bullet$ We denote the standard norm on $H_{0}^{1}(\Omega)$ by $\displaystyle\|u\|:=\Big(\int_{\Omega}|\nabla u|^{2}dx\Big)^{\frac{1}{2}}$  and  write $|\cdot|_{q}$ for the $L^{q}(\Omega)$-norm for $q\in[1,\infty]$.\\
\noindent $\bullet$  We always assume $\Omega$ is a smooth bounded domain of $\mathbb{R}^N$ with $0\in\Omega$. \\
\noindent $\bullet$ Let $E$ be a real Hilbert space and $I:E \to \R$ a functional of class ${C}^1$.
 We say that $(u_n)\subset E$ is a  Palais-Smale ($(PS)$ for short) sequence at $c$ for $I$ if $(u_n)$ satisfies
$$
I(u_n)\to c \,\,\, \mbox{and} \,\,\,\, I'(u_n)\to0,  \,\,\, \mbox{as} \,\,\, n\to\infty.
$$
Moreover, $I$ satisfies the $(PS)$ condition at $c$, if any $(PS)$ sequence at $c$ possesses a convergent subsequence.\vspace{5mm}\\

This paper is organized as follows: In Section 2, we investigate the Hardy-Littlewood-Sobolev critical Choquard equation perturbed by a subcritical superlinear local term. In Section 3, we consider the critical Choquard equation perturbed by a subcritical superlinear nonlocal term. In Section 4, we investigate the combined effects of the superlinear and sublinear nonlinearities on the existence and multiplicity of the critical Choquard equation. Finally, we study the existence of infinitely many solutions of the critical Choquard equation under the effect of a sublinear term.

\section{\large Perturbation with a superlinear local term}
In this section we will study the existence of solutions for the critical Choquard equation with superlinear local perturbation, i.e.
$$
\left\{\begin{array}{l}
\displaystyle-\Delta u
=\Big(\int_{\Omega}\frac{|u|^{2_{\mu}^{\ast}}}{|x-y|^{\mu}}dy\Big)|u|^{2_{\mu}^{\ast}-2}u+\lambda u^q\hspace{4.14mm}\mbox{in}\hspace{1.14mm} \Omega,\\
\displaystyle u\in H_{0}^{1}(\Omega),
\end{array}
\right.
$$
where   $1<q<2^{\ast}-1$.

By introducing the energy functional by
$$
J_{\lambda}(u)=\frac{1}{2}\int_{\Omega}|\nabla u|^{2}dx-\frac{\lambda}{q+1}\int_{\Omega} u^{q+1}dx-\frac{1}{2\cdot2_{\mu}^{\ast}}\int_{\Omega}\int_{\Omega}\frac{|u(x)|^{2_{\mu}^{\ast}}|u(y)|^{2_{\mu}^{\ast}}}
{|x-y|^{\mu}}dxdy,
$$
we check that the functional $J_{\lambda}$ satisfies the Mountain-Pass Geometry, that is
\begin{lem} \label{MP1}If $1<q<2^{\ast}-1$ and $\lambda>0$, then, the functional $J_{\lambda}$ satisfies the following properties:
\begin{itemize}
\item [(1).] There exist $\alpha,\rho>0$ such that $J_{\lambda}(u)\geq\alpha$ for $\|u\|=\rho$.
\item[(2).] There exists $e\in H_{0}^{1}(\Omega)$ with $\|e\|>\rho$ such that $J_{\lambda}(e)<0$.
\end{itemize}
\end{lem}
\begin{proof} (1). By the Sobolev embedding and the Hardy-Littlewood-Sobolev inequality, for all $u\in H_{0}^{1}(\Omega)\backslash\ \{0\}$ we have
$$
J_{\lambda}(u)\geq\frac{1}{2}\|u\|^{2}-\frac{\lambda}{q+1}C_{1}\|u\|^{q+1}-\frac{1}{2\cdot2_{\mu}^{\ast}}C_{0}\|u\|^{2(\frac{2N-\mu}{N-2})}.
$$
Since $2<q+1<2\cdot2_{\mu}^{\ast}$, we can choose some $\rho$ small enough such that $J_{\lambda}(u)\geq\alpha>0$ for all $u$ satisfying $\|u\|=\rho$.

(2). For some $u_{1}\in H_{0}^{1}(\Omega)\backslash\ \{0\}$, we have
$$
J_{\lambda}(tu_{1})=\frac{t^{2}}{2}\int_{\Omega}|\nabla u_{1}|^{2}dx-\frac{\lambda t^{q+1}}{q+1}\int_{\Omega} u_{1}^{q+1}dx-\frac{t^{2\cdot2_{\mu}^{\ast}}}{2\cdot2_{\mu}^{\ast}}
\int_{\Omega}\int_{\Omega}\frac{|u_{1}(x)|^{2_{\mu}^{\ast}}|u_{1}(y)|^{2_{\mu}^{\ast}}}
{|x-y|^{\mu}}dxdy<0
$$
for $t>0$ large enough. Hence, we can take an $e:=t_{1}u_{1}$ for some $t_{1}>0$ and the conclusion (2) follows.
\end{proof}

\begin{lem} \label{WSo}
Let $1<q<2^{\ast}-1$, $\lambda>0$. If $\{u_{n}\}$ is a $(PS)_c$ sequence of $J_{\lambda}$, then
$\{u_{n}\}$ is bounded. Let $u_0\in H_{0}^{1}(\Omega)$ be the weak limit of $\{u_{n}\}$, then $u_0$ is a weak solution of problem \eqref{CCE1}.
\end{lem}
\begin{proof} It is easy to see that there exists $C_{1}>0$ such that
$$
|J_{\lambda}(u_{n})|\leq C_{1},\ \ \
|\langle J_{\lambda}^{'}(u_{n}),\frac{u_{n}}{\|u_{n}\|}\rangle|\leq C_{1}.
$$
Then  we have
$$
\aligned
J_{\lambda}(u_{n})-\frac{1}{q+1}\langle &J_{\lambda}^{'}(u_{n}),u_{n}\rangle\\
&=(\frac{1}{2}-\frac{1}{q+1})\int_{\Omega}|\nabla u_{n}|^{2}dx+(\frac{1}{q+1}-\frac{1}{2\cdot2_{\mu}^{\ast}})\int_{\Omega}\int_{\Omega}\frac{|u_{n}(x)|^{2_{\mu}^{\ast}}|u_{n}(y)|^{2_{\mu}^{\ast}}}{|x-y|^{\mu}}dxdy\\
&\leq C_{1}(1+\|u_{n}\|).
\endaligned
$$
Since $2<q+1<2\cdot2_{\mu}^{\ast}$, we know that $\{u_{n}\}$ is bounded in $H_{0}^{1}(\Omega)$.

Up to a subsequence, there exists $u_0\in H_{0}^{1}(\Omega)$ such that $u_{n}\rightharpoonup u_0$ in $H_{0}^{1}(\Omega)$ and
$
u_{n}\rightharpoonup u_0
$
in $L^{2^{*}}(\Omega)$  as $n\rightarrow+\infty$. Then
$$
|u_{n}|^{2_{\mu}^{*}}\rightharpoonup |u_0|^{2_{\mu}^{*}} \hspace{3.14mm} \mbox{in} \hspace{3.14mm} L^{\frac{2N}{2N-\mu}}(\Omega)
$$
as $n\rightarrow+\infty$. By the Hardy-Littlewood-Sobolev inequality,
the Riesz potential defines a linear continuous map from  $L^{\frac{2N}{2N-\mu}}(\Omega)$ to $L^{\frac{2N}{\mu}}(\Omega)$,  hence
$$
|x|^{-\mu}\ast|u_{n}|^{2_{\mu}^{*}}\rightharpoonup |x|^{-\mu}\ast|u_0|^{2_{\mu}^{*}} \hspace{3.14mm} \mbox{in} \hspace{3.14mm} L^{\frac{2N}{\mu}}(\Omega)
$$
as $n\rightarrow+\infty$. Combining this with the fact that
$$
|u_{n}|^{2_{\mu}^{\ast}-2}u_{n}\rightharpoonup |u_0|^{2_{\mu}^{\ast}-2}u_0 \hspace{3.14mm} \mbox{in} \hspace{3.14mm} L^{\frac{2N}{N-\mu+2}}(\Omega),
$$
as $n\rightarrow+\infty$, we have
$$
(|x|^{-\mu}\ast|u_{n}|^{2_{\mu}^{*}})|u_{n}|^{2_{\mu}^{\ast}-2}u_{n}\rightharpoonup (|x|^{-\mu}\ast|u_0|^{2_{\mu}^{*}})|u_0|^{2_{\mu}^{\ast}-2}u_0 \hspace{3.14mm} \mbox{in} \hspace{3.14mm} L^{\frac{2N}{N+2}}(\Omega)
$$
as $n\rightarrow+\infty$. Since, for any $\varphi\in H_{0}^{1}(\Omega)$,
$$
0\leftarrow\langle J_{\lambda}^{'}(u_{n}),\varphi\rangle=\int_{\Omega}\nabla u_{n}\nabla\varphi dx
-\lambda\int_{\Omega}u_{n}^{q}\varphi dx-
\int_{\Omega}\int_{\Omega}\frac{|u_{n}(x)|^{2_{\mu}^{\ast}}|u_{n}(y)|^{2_{\mu}^{\ast}-2}u_{n}(y)\varphi(y)}{|x-y|^{\mu}}dxdy.
$$
Passing to the limit as $n\rightarrow+\infty$, we obtain
$$
\int_{\Omega}\nabla u_0\nabla\varphi dx
-\lambda\int_{\Omega}u_{0}^{q}\varphi dx-
\int_{\Omega}\int_{\Omega}\frac{|u_0(x)|^{2_{\mu}^{\ast}}|u_0(y)|^{2_{\mu}^{\ast}-2}u_0(y)\varphi(y)}{|x-y|^{\mu}}dxdy=0
$$
for any $\varphi\in H_{0}^{1}(\Omega)$, which means $u_0$ is a weak solution of the problem \eqref{CCE1}.

Finally, taking $\varphi=u_0\in H_{0}^{1}(\Omega)$ as a test function in equation \eqref{CCE1}, we have
$$
\int_{\Omega}|\nabla u_0|^{2}dx=
\lambda\int_{\Omega}u_0^{q+1}dx+
\int_{\Omega}\int_{\Omega}\frac{|u_0(x)|^{2_{\mu}^{\ast}}|u_0(y)|^{2_{\mu}^{\ast}}}{|x-y|^{\mu}}dxdy,
$$
and so, for  $1<q<2^{\ast}-1$,
$$
J_{\lambda}(u_0)=(\frac{\lambda}{2}-\frac{\lambda}{q+1})\int_{\Omega}u_0^{q+1}dx+\frac{N+2-\mu}{4N-2\mu}\int_{\Omega}\int_{\Omega}\frac{|u_0(x)|^{2_{\mu}^{\ast}}|u_0(y)|^{2_{\mu}^{\ast}}}
{|x-y|^{\mu}}dxdy\geq0.
$$
\end{proof}

The following Brezis-Lieb type lemma for the nonlocal term is proved in \cite{AC}  (the subcritical case) and \cite{GY} (the critical case).
\begin{lem} \label{BLN} Let $N\geq3$, $0<\mu<N$ and $(2N-\mu)/2N\leq p\leq2_{\mu}^{\ast}$. If $\{u_{n}\}$ is a bounded sequence in $L^{\frac{2N}{N-2}}(\Omega)$ such that $u_{n}\rightarrow u$ almost everywhere in $\Omega$ as $n\rightarrow\infty$, then the following hold,
$$
\int_{\Omega}(|x|^{-\mu}\ast |u_{n}|^{p})|u_{n}|^{p}dx-\int_{\Omega}(|x|^{-\mu}\ast |u_{n}-u|^{p})|u_{n}-u|^{p}dx\rightarrow\int_{\Omega}(|x|^{-\mu}\ast |u|^{p})|u|^{p}dx
$$
as $n\rightarrow\infty$.
\end{lem}

In the next Lemma we prove a convergence criterion for the $ (PS) _c$ sequences which will play an important role in applying the critical point theorems.

\begin{lem}\label{ConPro} Let $1<q<2^{\ast}-1$ and $\lambda>0$. If $\{u_{n}\}$ is a $(PS)_c$ sequence of $J_{\lambda}$ with
\begin{equation}\label{C1}
c<\frac{N+2-\mu}{4N-2\mu}S_{H,L}^{\frac{2N-\mu}{N+2-\mu}},
\end{equation}
then $\{u_{n}\}$ has a convergent  subsequence.
\end{lem}
\begin{proof}
Let $u_0$ be the weak limit of  $\{u_{n}\}$ obtained in  Lemma \ref{WSo} and define  $v_{n}:=u_{n}-u_0$, then we know $v_{n}\rightharpoonup0$ in $H_{0}^{1}(\Omega)$ and $v_{n}\rightarrow 0$ a.e. in $\Omega$. Moreover, by the Br\'{e}zis-Lieb Lemma  in \cite{BL1} and Lemma \ref{BLN}, we know
$$
\int_{\Omega}|\nabla u_{n}|^{2}dx=\int_{\Omega}|\nabla v_{n}|^{2}dx
+\int_{\Omega}|\nabla u_0|^{2}dx+o_n(1),
$$
$$
\int_{\Omega}|u_{n}|^{q+1}dx=\int_{\Omega}|v_{n}|^{q+1}dx+\int_{\Omega}|u_0|^{q+1}dx+o_n(1)
$$
and
$$
\int_{\Omega}\int_{\Omega}\frac{|u_{n}(x)|^{2_{\mu}^{\ast}}|u_{n}(y)|^{2_{\mu}^{\ast}}}
{|x-y|^{\mu}}dxdy=\int_{\Omega}\int_{\Omega}\frac{|v_{n}(x)|^{2_{\mu}^{\ast}}|v_{n}(y)|^{2_{\mu}^{\ast}}}
{|x-y|^{\mu}}dxdy+\int_{\Omega}\int_{\Omega}\frac{|u_0(x)|^{2_{\mu}^{\ast}}|u_0(y)|^{2_{\mu}^{\ast}}}
{|x-y|^{\mu}}dxdy+o_n(1).
$$
Consequently, we have
\begin{equation}\label{C2}
\aligned
c\leftarrow J_{\lambda}(u_{n})&=\frac{1}{2}\int_{\Omega}|\nabla u_{n}|^{2}dx-\frac{\lambda}{q+1}\int_{\Omega} u_{n}^{q+1}dx-\frac{1}{2\cdot2_{\mu}^{\ast}}\int_{\Omega}\int_{\Omega}\frac{|u_{n}(x)|^{2_{\mu}^{\ast}}|u_{n}(y)|^{2_{\mu}^{\ast}}}
{|x-y|^{\mu}}dxdy\\
&=\frac{1}{2}\int_{\Omega}|\nabla v_{n}|^{2}dx-\frac{\lambda}{q+1}\int_{\Omega}v_{n}^{q+1}dx+\frac{1}{2}\int_{\Omega}|\nabla u_0|^{2}dx-\frac{\lambda}{q+1}\int_{\Omega} u_0^{q+1}dx\\
&\hspace{7mm}-\frac{1}{2\cdot2_{\mu}^{\ast}}\int_{\Omega}\int_{\Omega}\frac{|v_{n}(x)|^{2_{\mu}^{\ast}}|v_{n}(y)|^{2_{\mu}^{\ast}}}
{|x-y|^{\mu}}dxdy-\frac{1}{2\cdot2_{\mu}^{\ast}}\int_{\Omega}\int_{\Omega}\frac{|u_0(x)|^{2_{\mu}^{\ast}}|u_0(y)|^{2_{\mu}^{\ast}}}
{|x-y|^{\mu}}dxdy+o_n(1)\\
&=J_{\lambda}(u_0)+\frac{1}{2}\int_{\Omega}|\nabla v_{n}|^{2}dx-\frac{\lambda}{q+1}\int_{\Omega} v_{n}^{q+1}dx-\frac{1}{2\cdot2_{\mu}^{\ast}}\int_{\Omega}\int_{\Omega}\frac{|v_{n}(x)|^{2_{\mu}^{\ast}}|v_{n}(y)|^{2_{\mu}^{\ast}}}
{|x-y|^{\mu}}dxdy+o_n(1)\\
&\geq \frac{1}{2}\int_{\Omega}|\nabla v_{n}|^{2}dx-\frac{1}{2\cdot2_{\mu}^{\ast}}\int_{\Omega}\int_{\Omega}\frac{|v_{n}(x)|^{2_{\mu}^{\ast}}|v_{n}(y)|^{2_{\mu}^{\ast}}}
{|x-y|^{\mu}}dxdy+o_n(1),
\endaligned
\end{equation}
since $J_{\lambda}(u_0)\geq 0$ and $\displaystyle\int_{\Omega}v_{n}^{q+1}dx\rightarrow 0$, as $n\rightarrow+\infty$.
Similarly, since  $\langle J_{\lambda}^{'}(u_0),u_0\rangle=0$, we have
\begin{equation}\label{C3}
\aligned
o_n(1)&= \langle J_{\lambda}^{'}(u_{n}),u_{n}\rangle\\
& =\int_{\Omega}|\nabla u_{n}|^{2}dx-\lambda\int_{\Omega} u_{n}^{q+1}dx-\int_{\Omega}\int_{\Omega}\frac{|u_{n}(x)|^{2_{\mu}^{\ast}}|u_{n}(y)|^{2_{\mu}^{\ast}}}
{|x-y|^{\mu}}dxdy\\
&=\int_{\Omega}|\nabla v_{n}|^{2}dx-\lambda\int_{\Omega} v_{n}^{q+1}dx+\int_{\Omega}|\nabla u_0|^{2}dx-\lambda\int_{\Omega} u_0^{q+1}dx\\
&\hspace{0.5cm}-\int_{\Omega}\int_{\Omega}\frac{|v_{n}(x)|^{2_{\mu}^{\ast}}|v_{n}(y)|^{2_{\mu}^{\ast}}}
{|x-y|^{\mu}}dxdy-\int_{\Omega}\int_{\Omega}\frac{|u_0(x)|^{2_{\mu}^{\ast}}|u_0(y)|^{2_{\mu}^{\ast}}}
{|x-y|^{\mu}}dxdy+o_n(1)\\
&=\langle J_{\lambda}^{'}(u_0),u_0\rangle+\int_{\Omega}|\nabla v_{n}|^{2}dx-\lambda\int_{\Omega}v_{n}^{q+1}dx-\int_{\Omega}\int_{\Omega}\frac{|v_{n}(x)|^{2_{\mu}^{\ast}}|v_{n}(y)|^{2_{\mu}^{\ast}}}
{|x-y|^{\mu}}dxdy+o_n(1)\\
&=\int_{\Omega}|\nabla v_{n}|^{2}dx-\int_{\Omega}\int_{\Omega}\frac{|v_{n}(x)|^{2_{\mu}^{\ast}}|v_{n}(y)|^{2_{\mu}^{\ast}}}
{|x-y|^{\mu}}dxdy+o_n(1).
\endaligned
\end{equation}
From \eqref{C3},  we know there exists a nonnegative constant $b$ such that
$$
\int_{\Omega}|\nabla v_{n}|^{2}dx\rightarrow b
$$
and
$$
\int_{\Omega}\int_{\Omega}\frac{|v_{n}(x)|^{2_{\mu}^{\ast}}|v_{n}(y)|^{2_{\mu}^{\ast}}}
{|x-y|^{\mu}}dxdy\rightarrow b
$$
as $n\rightarrow+\infty$.
Thus from \eqref{C2}, we obtain
\begin{equation}\label{C4}
c\geq \frac{N+2-\mu}{4N-2\mu}b.
\end{equation}
By the definition of the best constant $S_{H,L}$ in \eqref{S1}, we have
$$
S_{H,L}\Big(\int_{\Omega}\int_{\Omega}\frac{|v_{n}(x)|^{2_{\mu}^{\ast}}|v_{n}(y)|^{2_{\mu}^{\ast}}}
{|x-y|^{\mu}}dxdy\Big)^{\frac{N-2}{2N-\mu}}\leq\int_{\Omega}|\nabla v_{n}|^{2}dx,
$$
which yields $b\geq S_{H,L}b^{\frac{N-2}{2N-\mu}}$. Thus we have either $b=0$ or $b\geq S_{H,L}^{\frac{2N-\mu}{N-\mu+2}}$. If $b\geq S_{H,L}^{\frac{2N-\mu}{N-\mu+2}}$,  then we obtain from \eqref{C4} that
$$
\frac{N+2-\mu}{4N-2\mu}S_{H,L}^{\frac{2N-\mu}{N-\mu+2}}\leq\frac{N+2-\mu}{4N-2\mu}b\leq c,
$$
which contradicts with the fact that $c<\frac{N+2-\mu}{4N-2\mu}S_{H,L}^{\frac{2N-\mu}{N+2-\mu}}$. Thus $b=0$, and
$$
\|u_{n}-u_0\|\rightarrow0
$$
as $n\rightarrow+\infty$. This completes the proof of Lemma \ref{ConPro}.
\end{proof}

\begin{lem}\label{EMP1} There exists $u_{\varepsilon}$ such that
\begin{equation}\label{C11}
\sup_{t\geq0}J_{\lambda}(tu_{\varepsilon})<\frac{N+2-\mu}{4N-2\mu}S_{H,L}^{\frac{2N-\mu}{N+2-\mu}}
\end{equation}
provided that either\\
(1). $N>\max\{\min\{\frac{2(q+3)}{q+1},2+\frac{\mu}{q+1}\},\frac{2(q+1)}{q}\}$ and $\lambda>0$, or\\
(2). $N\leq\max\{\min\{\frac{2(q+3)}{q+1},2+\frac{\mu}{q+1}\},\frac{2(q+1)}{q}\}$ and $\lambda$ is sufficiently large.
\end{lem}
\begin{proof} From Theorem 1.42 in \cite{Wi}, we know
$U(x)=\frac{[N(N-2)]^{\frac{N-2}{4}}}{(1+|x|^{2})^{\frac{N-2}{2}}}$  is a minimizer for $S$, the best Sobolev constant. By Proposition \ref{ExFu}, we know that $U(x)$ is also a minimizer for $S_{H,L}$.
Assume that $B_{\delta}\subset\Omega\subset B_{2\delta}$ and let $\psi\in C_{0}^{\infty}(\Omega)$ be such that
$$
\left\{\begin{array}{l}
\displaystyle \psi(x)=\left\{\begin{array}{l}
\displaystyle 1 \hspace{13.14mm} \mbox{if}\hspace{2.14mm} x\in B_{\delta},\\
\displaystyle 0 \hspace{13.14mm} \mbox{if} \hspace{2.14mm}x\in \mathbb{R}^N \setminus\Omega,\\
\end{array}
\right.\\
\displaystyle 0\leq\psi(x)\leq1 \hspace{19.14mm} \forall x\in \mathbb{R}^N ,\\
\displaystyle |D\psi(x)|\leq C=const \hspace{13.14mm} \forall x\in \mathbb{R}^N .\\
\end{array}
\right.
$$
We define, for $\varepsilon>0$,
\begin{equation}\label{EFV}
\aligned
U_{\varepsilon}(x)&:=\varepsilon^{\frac{2-N}{2}}U(\frac{x}{\varepsilon}),\\
u_{\varepsilon}(x)&:=\psi(x)U_{\varepsilon}(x).\\
\endaligned
\end{equation}
From \cite{GY}, we know that
\begin{equation}\label{C9}
\int_{\Omega}|\nabla u_{\varepsilon}|^{2}dx
=C(N,\mu))^{\frac{N-2}{2N-\mu}\cdot\frac{N}{2}}S_{H,L}^{\frac{N}{2}}+O(\varepsilon^{N-2})
\end{equation}
and
\begin{equation}\label{C10}
\int_{\Omega}\int_{\Omega}\frac{|u_{\varepsilon}(x)|^{2_{\mu}^{\ast}}|u_{\varepsilon}(y)|^{2_{\mu}^{\ast}}}
{|x-y|^{\mu}}dxdy
\geq C(N,\mu)^{\frac{N}{2}}S_{H,L}^{\frac{2N-\mu}{2}}-O(\varepsilon^{N-\frac{\mu}{2}}).
\end{equation}
\par
\textbf{Case 1. } $N>\max\{\min\{\frac{2(q+3)}{q+1},2+\frac{\mu}{q+1}\},\frac{2(q+1)}{q}\}$.

First, by the proof of Lemma 4.1 in \cite{DH},  since $q>1$ and $N>\frac{2(q+1)}{q}$ we know  $N<(N-2)(q+1)$, thus we have
\begin{equation}\label{C12}
\int_{\Omega}u_{\varepsilon}^{q+1}dx=O(\varepsilon^{N-\frac{(N-2)(q+1)}{2}})
+O(\varepsilon^{\frac{(N-2)(q+1)}{2}})=O(\varepsilon^{N-\frac{(N-2)(q+1)}{2}}),
\end{equation}
for $\varepsilon>0$ sufficiently small.

Using the estimates in \eqref{C9}, \eqref{C10} and \eqref{C12}, we have
$$\aligned
J_{\lambda}(tu_{\varepsilon})&=\frac{t^{2}}{2}\int_{\Omega}|\nabla u_{\varepsilon}|^{2}dx-\frac{ \lambda t^{q+1}}{q+1}\int_{\Omega} |u_{\varepsilon}|^{q+1}dx-\frac{t^{2\cdot2_{\mu}^{\ast}}}{2\cdot2_{\mu}^{\ast}}
\int_{\Omega}\int_{\Omega}\frac{|u_{\varepsilon}(x)|^{2_{\mu}^{\ast}}|u_{\varepsilon}(y)|^{2_{\mu}^{\ast}}}
{|x-y|^{\mu}}dxdy\\
&\leq\frac{t^{2}}{2}\big(C(N,\mu)^{\frac{N-2}{2N-\mu}\cdot\frac{N}{2}}S_{H,L}^{\frac{N}{2}}+O(\varepsilon^{N-2})\big)
-\lambda t^{q+1}O(\varepsilon^{N-\frac{(N-2)(q+1)}{2}})\\
&\hspace{7mm}-\frac{t^{2\cdot2_{\mu}^{\ast}}}{2\cdot2_{\mu}^{\ast}}\big(C(N,\mu)^{\frac{N}{2}}S_{H,L}^{\frac{2N-\mu}{2}}-O(\varepsilon^{N-\frac{\mu}{2}})\big)\\
&:=g(t).
\endaligned
$$
It is clear that $g(t)\rightarrow -\infty$ as $t\rightarrow+\infty$. It follows that there exists $t_{\varepsilon}>0$ such that $\sup_{t>0}g(t)$ is attained at $t_{\varepsilon}$. Differentiating $g(t)$ and equaling to zero, we obtain that
$$
t_{\varepsilon}\big(C(N,\mu)^{\frac{N-2}{2N-\mu}\cdot\frac{N}{2}}S_{H,L}^{\frac{N}{2}}+O(\varepsilon^{N-2})\big)
-\lambda t^{q}O(\varepsilon^{N-\frac{(N-2)(q+1)}{2}})
-t_{\varepsilon}^{2\cdot2_{\mu}^{\ast}-1}\big(C(N,\mu)^{\frac{N}{2}}S_{H,L}^{\frac{2N-\mu}{2}}-O(\varepsilon^{N-\frac{\mu}{2}})\big)
=0
$$
and so
$$
t_{\varepsilon}<\Big(\frac{C(N,\mu)^{\frac{N-2}{2N-\mu}\cdot\frac{N}{2}}S_{H,L}^{\frac{N}{2}}+O(\varepsilon^{N-2})}
{C(N,\mu)^{\frac{N}{2}}S_{H,L}^{\frac{2N-\mu}{2}}-O(\varepsilon^{N-\frac{\mu}{2}})}\Big)^{\frac{1}{2\cdot2_{\mu}^{\ast}-2}}:=S_{H,L}(\varepsilon)
$$
and there exists $t_{0}>0$ independent of $\vr$ such that for $\varepsilon>0$ small enough
$$
t_{\varepsilon}>t_{0}.
$$

Notice that the function
$$
t\mapsto \frac{t^{2}}{2}\big(C(N,\mu)^{\frac{N-2}{2N-\mu}\cdot\frac{N}{2}}S_{H,L}^{\frac{N}{2}}+O(\varepsilon^{N-2})\big)
-\frac{t^{2\cdot2_{\mu}^{\ast}}}{2\cdot2_{\mu}^{\ast}}\big(C(N,\mu)^{\frac{N}{2}}S_{H,L}^{\frac{2N-\mu}{2}}-O(\varepsilon^{N-\frac{\mu}{2}})\big)
$$
is increasing on $[0,S_{H,L}(\varepsilon)]$, we have
$$\aligned
\max_{t\geq0}J_{\lambda}(tu_{\varepsilon})
&\leq\frac{N+2-\mu}{4N-2\mu}\Big(\frac{C(N,\mu)^{\frac{N-2}{2N-\mu}\cdot\frac{N}{2}}S_{H,L}^{\frac{N}{2}}+O(\varepsilon^{N-2})}
{\Big(C(N,\mu)^{\frac{N}{2}}S_{H,L}^{\frac{2N-\mu}{2}}-O(\varepsilon^{N-\frac{\mu}{2}})\Big)^{\frac{N-2}{2N-\mu}}}\Big)^{\frac{2N-\mu}{N+2-\mu}}
-O(\varepsilon^{N-\frac{(N-2)(q+1)}{2}})\\
&\leq\frac{N+2-\mu}{4N-2\mu}S_{H,L}^{\frac{2N-\mu}{N+2-\mu}}+O(\varepsilon^{\min\{N-2,N-\frac{\mu}{2}\}})
-O(\varepsilon^{N-\frac{(N-2)(q+1)}{2}})\\
&<\frac{N+2-\mu}{4N-2\mu}S_{H,L}^{\frac{2N-\mu}{N+2-\mu}},\\
\endaligned
$$
thanks to $t_{0}<t_{\varepsilon}<S_{H,L}(\varepsilon)$, \eqref{C12} and $N>\min\{\frac{2(q+3)}{q+1},2+\frac{\mu}{q+1}\}$.

\textbf{Case 2. } $N\leq\max\{\min\{\frac{2(q+3)}{q+1},2+\frac{\mu}{q+1}\},\frac{2(q+1)}{q}\}$.

For any fixed $\vr$ in \eqref{EFV}, from
$$
J_{\lambda}(tu_{\varepsilon})\rightarrow -\infty
$$
as $t\rightarrow+\infty$, we have that $\max_{t\geq0}J_{\lambda}(tu_{\varepsilon})$ is attained at some $t_{\lambda}>0$ and $t_{\lambda}$ satisfies
$$
t_{\lambda}\int_{\Omega}|\nabla u_{\varepsilon}|^{2}dx= \lambda t_{\lambda}^{q}\int_{\Omega} |u_{\varepsilon}|^{q+1}dx+t_{\lambda}^{2\cdot2_{\mu}^{\ast}-1}\int_{\Omega}\int_{\Omega}
\frac{|u_{\varepsilon}(x)|^{2_{\mu}^{\ast}}|u_{\varepsilon}(y)|^{2_{\mu}^{\ast}}}
{|x-y|^{\mu}}dxdy,
$$
that is
$$
\int_{\Omega}|\nabla u_{\varepsilon}|^{2}dx= \lambda t_{\lambda}^{q-1}\int_{\Omega} |u_{\varepsilon}|^{q+1}dx+t_{\lambda}^{2\cdot2_{\mu}^{\ast}-2}\int_{\Omega}\int_{\Omega}
\frac{|u_{\varepsilon}(x)|^{2_{\mu}^{\ast}}|u_{\varepsilon}(y)|^{2_{\mu}^{\ast}}}
{|x-y|^{\mu}}dxdy,
$$
thanks to $\frac{\partial J_{\lambda}(tu_{\varepsilon})}{\partial t}|_{t=t_{\lambda}}=0$. Thus $t_{\lambda}\rightarrow0$ as $\lambda\rightarrow+\infty$. Then,
$$
\max_{t\geq0}J_{\lambda}(tu_{\varepsilon})=\frac{t_{\lambda}^{2}}{2}\int_{\Omega}|\nabla u_{\varepsilon}|^{2}dx-\frac{\lambda t_{\lambda}^{q+1}}{q+1}\int_{\Omega} |u_{\varepsilon}|^{q+1}dx-\frac{t_{\lambda}^{2\cdot2_{\mu}^{\ast}}}{2\cdot2_{\mu}^{\ast}}\int_{\Omega}\int_{\Omega}
\frac{|u_{\varepsilon}(x)|^{2_{\mu}^{\ast}}|u_{\varepsilon}(y)|^{2_{\mu}^{\ast}}}
{|x-y|^{\mu}}dxdy\rightarrow0
$$
as $\lambda\rightarrow+\infty$, which easily yields the desired conclusion for the case $N\leq\max\{\min\{\frac{2(q+3)}{q+1},2+\frac{\mu}{q+1}\},\frac{2(q+1)}{q}\}$.
\end{proof}

\noindent
{\bf Proof of Theorem \ref{EXS}.}
By Lemma \ref{MP1} and the Mountain Pass Theorem without $(PS)$ condition (cf. \cite{Wi}), there
exists a $(PS)$ sequence $\{u_{n}\}$ such that $J_{\lambda}(u_{n})\rightarrow c$ and $ J_{\lambda}^{'}(u_{n})\rightarrow0$ in $H_{0}^{1}(\Omega)^{-1}$ at the minimax level
$$
c=\inf\limits_{\gamma\in\Gamma}\max\limits_{t\in[0,1]}J_{\lambda}(\gamma(t))>0,
$$
where
$$
\Gamma:=\{\gamma\in C([0,1],H_{0}^{1}(\Omega)):\gamma(0)=0,J_{\lambda}(\gamma(1))<0\}.
$$
From Lemma \ref{EMP1} and the definition of $c$, we know $c<\frac{N+2-\mu}{4N-2\mu}S_{H,L}^{\frac{2N-\mu}{N+2-\mu}}$, provided that either
\begin{itemize}
\item[(1).] $N>\max\{\min\{\frac{2(q+3)}{q+1},2+\frac{\mu}{q+1}\},\frac{2(q+1)}{q}\}$ and $\lambda>0$, or
\item[(2).] $N\leq\max\{\min\{\frac{2(q+3)}{q+1},2+\frac{\mu}{q+1}\},\frac{2(q+1)}{q}\}$ and $\lambda$ is sufficiently large.
\end{itemize}
Applying Lemma \ref{ConPro}, we know  $\{u_{n}\}$ contains a convergent subsequence. And so, we have $J_{\lambda}$ has a critical value $c\in(0, \frac{N+2-\mu}{4N-2\mu}S_{H,L}^{\frac{2N-\mu}{N+2-\mu}})$ and thus the problem \eqref{CCE1} has a nontrivial solution.  $\hfill{} \Box$

\section{\large  Perturbation with a superlinear nonlocal term}
In this section we will study the existence of solutions in the case of a nonlocal perturbation,
$$
\left\{\begin{array}{l}
\displaystyle-\Delta u
=\Big(\int_{\Omega}\frac{|u|^{2_{\mu}^{\ast}}}{|x-y|^{\mu}}dy\Big)|u|^{2_{\mu}^{\ast}-2}u+\lambda \Big(\int_{\Omega}\frac{|u|^{q}}{|x-y|^{\mu}}dy\Big)|u|^{q-2}u\hspace{4.14mm}\mbox{in}\hspace{1.14mm} \Omega,\\
\displaystyle u\in H_{0}^{1}(\Omega).
\end{array}
\right.
$$ Since the problem is set in a bounded domain, the Sobolev imbedding and the Hardy-Littlewood-Sobolev inequality imply that the integral
$$
\int_{\Omega}\int_{\Omega}\frac{|u(x)|^{q}|u(y)|^{q}}{|x-y|^{\mu}}dxdy
$$
is well defined if
$$
\frac{2N-\mu}{2N}\leq q\leq\frac{2N-\mu}{N-2}.
$$
Thus, associated to the equation  \eqref{CCE2},  we can introduce the energy functional
$$
J_{\lambda}(u)=\frac{1}{2}\int_{\Omega}|\nabla u|^{2}dx-\frac{1}{2\cdot2_{\mu}^{\ast}}\int_{\Omega}
\int_{\Omega}\frac{|u(x)|^{2_{\mu}^{\ast}}|u(y)|^{2_{\mu}^{\ast}}}
{|x-y|^{\mu}}dxdy-\frac{\lambda}{2q}\int_{\Omega}
\int_{\Omega}\frac{|u(x)|^{q}|u(y)|^{q}}
{|x-y|^{\mu}}dxdy,
$$
which belongs to $C^{1}(H_{0}^{1}(\Omega),\R)$,  and so $u$ is a weak solution of \eqref{CCE2} if and only if $u$ is a critical point of functional $J_{\lambda}$.

Similar to the proof of Lemma \ref{MP1} and Lemma \ref{WSo}, we have the following conclusions.
\begin{lem}\label{MP2} If $1<q<2_{\mu}^{\ast}$ and $\lambda>0$, then, the functional $J_{\lambda}$ satisfies the following properties:
\begin{itemize}
\item[(1).] There exist $\alpha,\rho>0$ such that $J_{\lambda}(u)\geq\alpha$ for $\|u\|=\rho$.
\item[(2).] There exists $e\in H_{0}^{1}(\Omega)$ with $\|e\|>\rho$ such that $J_{\lambda}(e)<0$.
\end{itemize}
\end{lem}

\begin{lem} \label{WSo2}
If $1<q<2_{\mu}^{\ast}$ and $\lambda>0$ and $\{u_{n}\}$ is a $(PS)_c$ sequence of $J_{\lambda}$, then
$\{u_{n}\}$ is bounded. Let $u_0\in H_{0}^{1}(\Omega)$ be the weak limit of $\{u_{n}\}$, then $u_0$ is a weak solution of problem \eqref{CCE2}.
\end{lem}

Since $1<q<2_{\mu}^{\ast}$  it is easy to see that
$$
J_{\lambda}(u_0)=(\frac{\lambda}{2}-\frac{\lambda}{2q})\int_{\Omega}\int_{\Omega}
\frac{|u_{0}(x)|^{q}|u_{0}(y)|^{q}}{|x-y|^{\mu}}dxdy+\frac{N+2-\mu}{4N-2\mu}\int_{\Omega}\int_{\Omega}\frac{|u_0(x)|^{2_{\mu}^{\ast}}|u_0(y)|^{2_{\mu}^{\ast}}}
{|x-y|^{\mu}}dxdy\geq0.
$$
\begin{lem}\label{ConPro2}If $1<q<2_{\mu}^{\ast}$,  $\lambda>0$ and $\{u_{n}\}$ is a $(PS)_c$ sequence of $J_{\lambda}$ with
\begin{equation}\label{D1}
c<\frac{N+2-\mu}{4N-2\mu}S_{H,L}^{\frac{2N-\mu}{N+2-\mu}}.
\end{equation}
Then $\{u_{n}\}$ has a convergent  subsequence.
\end{lem}
\begin{proof}
Let $u_0$ be the weak limit of  $\{u_{n}\}$ obtained in  Lemma \ref{WSo2} and define  $v_{n}:=u_{n}-u_0$, then we know $v_{n}\rightharpoonup0$ in $H_{0}^{1}(\Omega)$ and $v_{n}\rightarrow 0$ a.e. in $\Omega$. Moreover, by Lemma \ref{BLN}, we know
$$
\int_{\Omega}\int_{\Omega}
\frac{|u_{n}(x)|^{q}|u_{n}(y)|^{q}}{|x-y|^{\mu}}dxdy=\int_{\Omega}\int_{\Omega}
\frac{|v_{n}(x)|^{q}|v_{n}(y)|^{q}}{|x-y|^{\mu}}dxdy+\int_{\Omega}\int_{\Omega}
\frac{|u_{0}(x)|^{q}|u_{0}(y)|^{q}}{|x-y|^{\mu}}dxdy+o_n(1).
$$
Similar to the
proof of Lemma \ref{ConPro}, we have
$$
c\leftarrow J_{\lambda}(u_{n})\geq \frac{1}{2}\int_{\Omega}|\nabla v_{n}|^{2}dx-\frac{1}{2\cdot2_{\mu}^{\ast}}\int_{\Omega}\int_{\Omega}\frac{|v_{n}(x)|^{2_{\mu}^{\ast}}|v_{n}(y)|^{2_{\mu}^{\ast}}}
{|x-y|^{\mu}}dxdy+o_n(1),
$$
since $J_{\lambda}(u_0)\geq 0$ and $$\int_{\Omega}\int_{\Omega}
\frac{|v_{n}(x)|^{q}|v_{n}(y)|^{q}}{|x-y|^{\mu}}dxdy\leq C(N,\mu)|v_{n}^{q}|_{\frac{2N}{2N-\mu}}^{2}\rightarrow 0,$$ as $n\rightarrow+\infty$.
Repeating the same arguments in the proof of Lemma \ref{ConPro}, we have
$$
\|u_{n}-u_0\|\rightarrow0
$$
as $n\rightarrow+\infty$. This ends the proof of Lemma \ref{ConPro2}.
\end{proof}

\begin{lem}\label{EMP2} Let $1<q<2_{\mu}^{\ast}$ and $u_{\varepsilon}$ as defined in \eqref{EFV}.
If one of the following conditions is holds:\\
(1) $N>\frac{2(q+1)-\mu}{q-1}$ and $\lambda>0$, \\
(2) $N\leq\frac{2(q+1)-\mu}{q-1}$ and $\lambda$ is sufficiently large,\\
then there exists $\vr$ such that
\begin{equation}\label{D2}
\sup_{t\geq0}J_{\lambda}(tu_{\varepsilon})<\frac{N+2-\mu}{4N-2\mu}S_{H,L}^{\frac{2N-\mu}{N+2-\mu}}.
\end{equation}
\end{lem}
\begin{proof}
\textbf{Case 1. } $N>\frac{2(q+1)-\mu}{q-1}$.

To estimate the convolution part, we know
\begin{equation}\label{D3}
\aligned
\int_{\Omega}\int_{\Omega}\frac{|u_{\varepsilon}(x)|^{q}
|u_{\varepsilon}(y)|^{q}}
{|x-y|^{\mu}}dxdy
&\geq \int_{B_{\delta}}\int_{B_{\delta}}\frac{|u_{\varepsilon}(x)|^{q}
|u_{\varepsilon}(y)|^{q}}{|x-y|^{\mu}}dxdy\\
&=\int_{B_{\delta}}\int_{B_{\delta}}\frac{|U_{\varepsilon}(x)|^{q}
|U_{\varepsilon}(y)|^{q}}{|x-y|^{\mu}}dxdy\\
&=\int_{\Omega}\int_{\Omega}\frac{|U_{\varepsilon}(x)|^{q}
|U_{\varepsilon}(y)|^{q}}{|x-y|^{\mu}}dxdy-2\int_{\Omega\setminus B_{\delta}}\int_{B_{\delta}}\frac{|U_{\varepsilon}(x)|^{q}|U_{\varepsilon}(y)|^{q}}{|x-y|^{\mu}}dxdy\\
&\hspace{12mm}-\int_{\Omega\setminus B_{\delta}}\int_{\Omega\setminus B_{\delta}}\frac{|U_{\varepsilon}(x)|^{q}|U_{\varepsilon}(y)|^{q}}{|x-y|^{\mu}}dxdy\\
&:=\A-2\B-\C,
\endaligned
\end{equation}
where
$$
\A=\int_{\Omega}\int_{\Omega}\frac{|U_{\varepsilon}(x)|^{q}
|U_{\varepsilon}(y)|^{q}}{|x-y|^{\mu}}dxdy,\ \
\B=\int_{\Omega\setminus B_{\delta}}\int_{B_{\delta}}\frac{|U_{\varepsilon}(x)|^{q}|U_{\varepsilon}(y)|^{q}}{|x-y|^{\mu}}dxdy
$$
and
$$
\C=\int_{\Omega\setminus B_{\delta}}\int_{\Omega\setminus B_{\delta}}\frac{|U_{\varepsilon}(x)|^{q}|U_{\varepsilon}(y)|^{q}}{|x-y|^{\mu}}dxdy.
$$
We are going to estimate $\A$, $\B$ and $\C$. By direct computation, we know, for $\varepsilon<1$,
\begin{equation}\label{D4}
\aligned
\A&=\varepsilon^{-(N-2)q}[N(N-2)]^{\frac{(N-2)q}{2}}\int_{\Omega}\int_{\Omega}\frac{1}
{(1+|\frac{x}{\varepsilon}|^{2})^{\frac{(N-2)q}{2}}|x-y|^{\mu}(1+|\frac{y}{\varepsilon}|^{2})^{\frac{(N-2)q}{2}}}dxdy\\
&\geq\varepsilon^{-(N-2)q}[N(N-2)]^{\frac{(N-2)q}{2}}\int_{B_{\delta}}\int_{B_{\delta}}\frac{1}
{(1+|\frac{x}{\varepsilon}|^{2})^{\frac{(N-2)q}{2}}|x-y|^{\mu}(1+|\frac{y}{\varepsilon}|^{2})^{\frac{(N-2)q}{2}}}dxdy\\
&=\varepsilon^{-(N-2)q}[N(N-2)]^{\frac{(N-2)q}{2}}\varepsilon^{2N-\mu}\int_{B_{\frac{\delta}{\varepsilon}}}\int_{B_{\frac{\delta}{\varepsilon}}}\frac{1}
{(1+|x|^{2})^{\frac{(N-2)q}{2}}|x-y|^{\mu}
(1+|y|^{2})^{\frac{(N-2)q}{2}}}dxdy\\
&\geq O(\varepsilon^{2N-\mu-(N-2)q})\int_{B_{\delta}}\int_{B_{\delta}}\frac{1}
{(1+|x|^{2})^{\frac{(N-2)q}{2}}|x-y|^{\mu}
(1+|y|^{2})^{\frac{(N-2)q}{2}}}dxdy\\
&=O(\varepsilon^{2N-\mu-(N-2)q}),
\endaligned
\end{equation}
\begin{equation}\label{D5}
\aligned
\B&=\varepsilon^{-(N-2)q}[N(N-2)]^{\frac{(N-2)q}{2}}\int_{\Omega\setminus B_{\delta}}\int_{B_{\delta}}\frac{1}
{(1+|\frac{x}{\varepsilon}|^{2})^{\frac{(N-2)q}{2}}|x-y|^{\mu}(1+|\frac{y}{\varepsilon}|^{2})^{\frac{(N-2)q}{2}}}dxdy\\
&=\varepsilon^{(N-2)q}[N(N-2)]^{\frac{(N-2)q}{2}}\int_{\Omega\setminus B_{\delta}}\int_{B_{\delta}}\frac{1}{(\varepsilon^{2}+|x|^{2})^{\frac{(N-2)q}{2}}|x-y|^{\mu}(\varepsilon^{2}+|y|^{2})^{\frac{(N-2)q}{2}}}dxdy\\
&\leq O(\varepsilon^{(N-2)q})\Big(\int_{\Omega\setminus B_{\delta}}\frac{1}
{(\varepsilon^{2}+|x|^{2})^{\frac{(N-2)qN}{2N-\mu}}}dx\Big)^{\frac{2N-\mu}{2N}}\Big(\int_{B_{\delta}}\frac{1}
{(\varepsilon^{2}+|y|^{2})^{\frac{(N-2)qN}{2N-\mu}}}dy\Big)^{\frac{2N-\mu}{2N}}\\
&=O(\varepsilon^{\frac{2N-\mu}{2}})\Big(\int_{0}^{\frac{\delta}{\varepsilon}}\frac{z^{N-1}}
{(1+z^{2})^{\frac{(N-2)qN}{2N-\mu}}}dz\Big)^{\frac{2N-\mu}{2N}}\\
&\leq O(\varepsilon^{\frac{2N-\mu}{2}})\Big(\int_{0}^{+\infty}\frac{z^{N-1}}
{(1+z^{2})^{\frac{(N-2)qN}{2N-\mu}}}dz\Big)^{\frac{2N-\mu}{2N}}.\\
\endaligned
\end{equation}
Since $N>\frac{2(q+1)-\mu}{q-1}>2+\frac{4-\mu}{2(q-1)}$ if $\mu<4$ and $N>2+\frac{4-\mu}{2(q-1)}$ if $\mu\geq4$, we know $\frac{2(N-2)qN}{2N-\mu}>N$, therefore
\begin{equation}\label{D52}
\B\leq O(\varepsilon^{\frac{2N-\mu}{2}}).
\end{equation}
\begin{equation}\label{D6}
\aligned
\C&=\varepsilon^{-(N-2)q}[N(N-2)]^{\frac{(N-2)q}{2}}\int_{\Omega\setminus B_{\delta}}\int_{\Omega\setminus B_{\delta}}\frac{1}
{(1+|\frac{x}{\varepsilon}|^{2})^{\frac{(N-2)q}{2}}|x-y|^{\mu}(1+|\frac{y}{\varepsilon}|^{2})^{\frac{(N-2)q}{2}}}dxdy\\
&=\varepsilon^{(N-2)q}[N(N-2)]^{\frac{(N-2)q}{2}}\int_{\Omega\setminus B_{\delta}}\int_{\Omega\setminus B_{\delta}}\frac{1}{(\varepsilon^{2}+|x|^{2})^{\frac{(N-2)q}{2}}|x-y|^{\mu}
(\varepsilon^{2}+|y|^{2})^{\frac{(N-2)q}{2}}}dxdy\\
&\leq\varepsilon^{(N-2)q}[N(N-2)]^{\frac{(N-2)q}{2}}\int_{\Omega\setminus B_{\delta}}\int_{\Omega\setminus B_{\delta}}\frac{1}{|x|^{(N-2)q}|x-y|^{\mu}
|y|^{(N-2)q}}dxdy\\
&=O(\varepsilon^{(N-2)q}).\\
\endaligned
\end{equation}
It follows from \eqref{D3}-\eqref{D6} that
\begin{equation}\label{D7}
\aligned
\int_{\Omega}\int_{\Omega}\frac{|u_{\varepsilon}(x)|^{q}
|u_{\varepsilon}(y)|^{q}}
{|x-y|^{\mu}}dxdy
&\geq O(\varepsilon^{2N-\mu-(N-2)q})-O(\varepsilon^{\frac{2N-\mu}{2}})-O(\varepsilon^{(N-2)q})\\
&=O(\varepsilon^{2N-\mu-(N-2)q})-O(\varepsilon^{\min\{\frac{2N-\mu}{2},(N-2)q\}}).
\endaligned
\end{equation}
By \eqref{C9}, \eqref{C10} and \eqref{D7}, we have
$$\aligned
J_{\lambda}(tu_{\varepsilon})&=\frac{t^{2}}{2}\int_{\Omega}|\nabla u_{\varepsilon}|^{2}dx-\frac{t^{2\cdot2_{\mu}^{\ast}}}{2\cdot2_{\mu}^{\ast}}
\int_{\Omega}\int_{\Omega}\frac{|u_{\varepsilon}(x)|^{2_{\mu}^{\ast}}|u_{\varepsilon}(y)|^{2_{\mu}^{\ast}}}
{|x-y|^{\mu}}dxdy-\frac{\lambda t^{2q}}{2q}\int_{\Omega}\int_{\Omega}\frac{|u_{\varepsilon}(x)|^{q}
|u_{\varepsilon}(y)|^{q}}
{|x-y|^{\mu}}dxdy\\
&\leq\frac{t^{2}}{2}\big(C(N,\mu)^{\frac{N-2}{2N-\mu}\cdot\frac{N}{2}}S_{H,L}^{\frac{N}{2}}+O(\varepsilon^{N-2})\big)
-\frac{t^{2\cdot2_{\mu}^{\ast}}}{2\cdot2_{\mu}^{\ast}}\big(C(N,\mu)^{\frac{N}{2}}S_{H,L}^{\frac{2N-\mu}{2}}-O(\varepsilon^{N-\frac{\mu}{2}})\big)\\
&\hspace{7mm}-\frac{t^{2q}}{2q}\big(O(\varepsilon^{2N-\mu-(N-2)q})-O(\varepsilon^{\min\{\frac{2N-\mu}{2},(N-2)q\}})\big)\\
&:=g(t).
\endaligned
$$
It is clear that $g(t)\rightarrow -\infty$ as $t\rightarrow+\infty$. It follows that there exists $t_{\varepsilon}>0$ such that $\sup_{t>0}g(t)$ is attained at $t_{\varepsilon}$. Differentiating $g(t)$ and equaling to zero, we obtain that
$$\aligned
t_{\varepsilon}\big(C(N,\mu)^{\frac{N-2}{2N-\mu}\cdot\frac{N}{2}}S_{H,L}^{\frac{N}{2}}+O(\varepsilon^{N-2})\big)
&-t_{\varepsilon}^{2\cdot2_{\mu}^{\ast}-1}\big(C(N,\mu)^{\frac{N}{2}}S_{H,L}^{\frac{2N-\mu}{2}}
-O(\varepsilon^{N-\frac{\mu}{2}})\big)\\
&-t^{2q-1}\big(O(\varepsilon^{2N-\mu-(N-2)q})-O(\varepsilon^{\min\{\frac{2N-\mu}{2},(N-2)q\}})\big)=0,
\endaligned$$
since  $N>\frac{2(q+1)-\mu}{q-1}$ and $N\geq3$ we know
\begin{equation}\label{EXP1}
{2N-\mu-(N-2)q}<\min\{\frac{2N-\mu}{2},(N-2)q\},
\end{equation}
which means $$\big(O(\varepsilon^{2N-\mu-(N-2)q})-O(\varepsilon^{\min\{\frac{2N-\mu}{2},(N-2)q\}})\big)\geq 0$$ if $\vr$ is small enough.
And so
$$
t_{\varepsilon}<\Big(\frac{C(N,\mu)^{\frac{N-2}{2N-\mu}\cdot\frac{N}{2}}S_{H,L}^{\frac{N}{2}}+O(\varepsilon^{N-2})}
{C(N,\mu)^{\frac{N}{2}}S_{H,L}^{\frac{2N-\mu}{2}}-O(\varepsilon^{N-\frac{\mu}{2}})}\Big)^{\frac{1}{2\cdot2_{\mu}^{\ast}-2}}:=S_{H,L}(\varepsilon)
$$
and there exists $t_{0}>0$ such that for $\varepsilon>0$ small enough
$$
t_{\varepsilon}>t_{0}.
$$
Notice that the function
$$
t\mapsto \frac{t^{2}}{2}(C(N,\mu)^{\frac{N-2}{2N-\mu}\cdot\frac{N}{2}}S_{H,L}^{\frac{N}{2}}+O(\varepsilon^{N-2}))
-\frac{t^{2\cdot2_{\mu}^{\ast}}}{2\cdot2_{\mu}^{\ast}}(C(N,\mu)^{\frac{N}{2}}S_{H,L}^{\frac{2N-\mu}{2}}
-O(\varepsilon^{N-\frac{\mu}{2}}))
$$
is increasing on $[0,S_{H,L}(\varepsilon)]$, thanks to $t_{0}<t_{\varepsilon}<S_{H,L}(\varepsilon)$ and \eqref{D7}, we have
$$\aligned
\max_{t\geq0}&J_{\lambda}(tu_{\varepsilon})\\
&\leq\frac{N+2-\mu}{4N-2\mu}\Big(\frac{C(N,\mu)^{\frac{N-2}{2N-\mu}\cdot\frac{N}{2}}S_{H,L}^{\frac{N}{2}}+O(\varepsilon^{N-2})}
{\Big(C(N,\mu)^{\frac{N}{2}}S_{H,L}^{\frac{2N-\mu}{2}}-O(\varepsilon^{N-\frac{\mu}{2}})\Big)^{\frac{N-2}{2N-\mu}}}\Big)^{\frac{2N-\mu}{N+2-\mu}}
-O(\varepsilon^{2N-\mu-(N-2)q})+O(\varepsilon^{\min\{\frac{2N-\mu}{2},(N-2)q\}})\\
&\leq\frac{N+2-\mu}{4N-2\mu}S_{H,L}^{\frac{2N-\mu}{N+2-\mu}}
+O(\varepsilon^{\min\{\frac{2N-\mu}{2},N-2\}})
-O(\varepsilon^{2N-\mu-(N-2)q})+O(\varepsilon^{\min\{\frac{2N-\mu}{2},(N-2)q\}}).\\
\endaligned
$$
The assumptions $N>\frac{2(q+1)-\mu}{q-1}$ and $1<q<2_{\mu}^{\ast}$ together with \eqref{EXP1} imply that
$$\aligned
\max_{t\geq0}&J_{\lambda}(tu_{\varepsilon})
<\frac{N+2-\mu}{4N-2\mu}S_{H,L}^{\frac{2N-\mu}{N+2-\mu}},
\endaligned
$$
if $\vr $ is small enough.

\textbf{Case 2. } $N\leq\frac{2(q+1)-\mu}{q-1}$.

For any fixed $\vr$ in \eqref{EFV},  assume that  $\max_{t\geq0}J_{\lambda}(tu_{\varepsilon})$ is attained at some $t_{\lambda}>0$, repeat the arguments in the proof of Lemma \ref{EMP1}, we know $t_{\lambda}\rightarrow0$ as $\lambda\rightarrow+\infty$ and $\max_{t\geq0}J_{\lambda}(tu_{\varepsilon})\rightarrow0$,
as $\lambda\rightarrow+\infty$, which leads to the conclusion for the case $N\leq\frac{2(q+1)-\mu}{q-1}$.
\end{proof}

\noindent
{\bf Proof of Theorem \ref{EXS2}.}
By Lemma \ref{MP2}, Lemma \ref{EMP2} and the Mountain Pass Theorem without $(PS)$ condition (cf. \cite{Wi}), there
exists a $(PS)_c$ sequence $\{u_{n}\}$ of $J_{\lambda}$ with
 $c<\frac{N+2-\mu}{4N-2\mu}S_{H,L}^{\frac{2N-\mu}{N+2-\mu}}$ if one of the following conditions is holds:\\
(1). $N>\frac{2(q+1)-\mu}{q-1}$ and $\lambda>0$,\\
(2).  $N\leq\frac{2(q+1)-\mu}{q-1}$ and $\lambda$ is sufficiently large.\\  Applying Lemma \ref{ConPro2}, we know  $\{u_{n}\}$ contains a convergent subsequence. And so, we have $J_{\lambda}$ has a critical value $c\in(0, \frac{N+2-\mu}{4N-2\mu}S_{H,L}^{\frac{2N-\mu}{N+2-\mu}})$ and problem \eqref{CCE2} has a nontrivial solution. $\hfill{} \Box$
\\
\\
\noindent
{\bf Proof of Theorem \ref{EXS3}.}
The proof of Theorem \ref{EXS3} is similar to that of Theorem \ref{EXS2} and the main difference is that the $(PS)_c$ condition holds below the critical level $\frac{1}{N}S^{\frac{N}{2}}$. From Lemma 1.46 of \cite{Wi}, we know that
\begin{equation}\label{D8}
\int_{\Omega}|\nabla u_{\varepsilon}|^{2}dx=S^{\frac{N}{2}}+O(\varepsilon^{N-2})
\end{equation}
and
\begin{equation}\label{D9}
\int_{\Omega}|u_{\varepsilon}|^{2^{\ast}}dx=S^{\frac{N}{2}}+O(\varepsilon^{N}).
\end{equation}
For the Case 1 of Lemma \ref{EMP2}, we have
$$\aligned
\max_{t\geq0}J_{\lambda}(tu_{\varepsilon})
&\leq\frac{1}{N}\Big(\frac{S^{\frac{N}{2}}+O(\varepsilon^{N-2})}
{\Big(S^{\frac{N}{2}}+O(\varepsilon^{N})\Big)^{\frac{2}{2^{\ast}}}}\Big)^{\frac{N}{2}}
-O(\varepsilon^{2N-\mu-(N-2)q})+O(\varepsilon^{\min\{\frac{2N-\mu}{2},(N-2)q\}})\\
&<\frac{1}{N}S^{\frac{N}{2}}
+O(\varepsilon^{N-2})
-O(\varepsilon^{2N-\mu-(N-2)q})+O(\varepsilon^{\min\{\frac{2N-\mu}{2},(N-2)q\}})\\
&<\frac{1}{N}S^{\frac{N}{2}},\\
\endaligned
$$
thanks to $N>\frac{2(q+1)-\mu}{q-1}$.  The rest  of the proof is omitted here.$\hfill{} \Box$

\section{\large An Ambrosetti-Brezis-Cerami type concave and convex result}
In this section we discuss the problem \eqref{SCCE2} with both suplinear and sublinear local perturbation, i.e.
$$
\left\{\begin{array}{l}
\displaystyle-\Delta u
=\Big(\int_{\Omega}\frac{|u|^{2_{\mu}^{\ast}}}{|x-y|^{\mu}}dy\Big)|u|^{2_{\mu}^{\ast}-2}u+u^{p}+\lambda u^{q}\hspace{4.14mm}\mbox{in}\hspace{1.14mm} \Omega,\\
\displaystyle u\in H_{0}^{1}(\Omega),
\end{array}
\right.
$$
 where $0<q<1$ and $1<p<2^{\ast}-1$. Then we may define
$$
J_{\lambda}(u)=\frac{1}{2}\int_{\Omega}|\nabla u|^{2}dx-\frac{\lambda}{q+1}\int_{\Omega} u^{q+1}dx-\frac{1}{p+1}\int_{\Omega} u^{p+1}dx-\frac{1}{2\cdot2_{\mu}^{\ast}}\int_{\Omega}\int_{\Omega}\frac{|u(x)|^{2_{\mu}^{\ast}}|u(y)|^{2_{\mu}^{\ast}}}
{|x-y|^{\mu}}dxdy.
$$

The proof of the main results in Theorem \ref{EXS4}  will be separated into several Lemmas. We begin with a standard comparison method as well as some ideas given in [\cite{ABC}, Lemma 3.1]. Let $\Lambda$ be defined by
$$
\Lambda=\sup\{\lambda>0: \mbox{problem \eqref{SCCE2} has a solution} \}.
$$

\begin{lem}\label{EAC1}
$0<\Lambda<\infty$.
\end{lem}
\begin{proof}
To prove that $\Lambda>0$ we use the sub- and supersolution technique to construct a solution for any small $\lambda$ by using some ideas from \cite{ABC, GaPa}. Let $(\lambda_{1},e_{1})$ be the first eigenvalue and a corresponding positive eigenfunction of the Laplacian in $\Omega$. We can obtain a subsolution by taking a small multiple of $e_{1}$. We denote it as $\varepsilon e_{1}$. Let function $v$ denote the solution of
$$
\left\{\begin{array}{l}
\displaystyle-\Delta v=1\ \ \mbox{in}\ \ \Omega,\\
\\
\displaystyle v=0\ \ \ \ \ \ \mbox{on}\ \  \partial\Omega.
\end{array}
\right.
$$
Since $0<q<1<p<2\cdot2_{\mu}^{\ast}-1$, we can find $\lambda_{0}>0$ such that for $0<\lambda\leq\lambda_{0}$ there exists $M(\lambda)>0$ satisfying
$$
M\geq\lambda M^{q}|v|_{\infty}^{q}+M^{p}|v|_{\infty}^{p}+c_{0}M^{2\cdot2_{\mu}^{\ast}-1}|v|_{\infty}^{2\cdot2_{\mu}^{\ast}-1},
$$
where $c_{0}=\int_{\Omega+\Omega}\frac{1}{|y|^{\mu}}dy$, $\Omega+\Omega:=\{x+y\in\R^{N}:x,y\in\Omega\}$. As a consequence, the function $Mv$ verifies
$$
M=-\Delta(Mv)\geq\lambda (Mv)^{q}+(Mv)^{p}+M^{2\cdot2_{\mu}^{\ast}-1}\int_{\Omega}
\frac{|v(y)|^{2_{\mu}^{\ast}}}{|x-y|^{\mu}}dy|v(x)|^{2_{\mu}^{\ast}-2}v(x)
$$
and so $Mv$ is a supersolution of \eqref{SCCE2}. It follows that \eqref{SCCE2} has a positive solution $\varepsilon e_{1}\leq u\leq Mv$ for $0<\lambda\leq\lambda_{0}$ and so $\Lambda>0$.

Using $e_{1}$ as a test function in \eqref{SCCE2}, we have that
\begin{equation}\label{F1}
\aligned
\lambda_{1}\int_{\Omega} ue_{1} dx&=\int_{\Omega}\int_{\Omega}\frac{|u(x)|^{2_{\mu}^{\ast}}
|u(y)|^{2_{\mu}^{\ast}-2}u(y)e_{1}(y)}{|x-y|^{\mu}}dxdy+\lambda\int_{\Omega}u^{q}e_{1} dx
+\int_{\Omega}u^{p}e_{1} dx\\
&\geq \lambda\int_{\Omega}u^{q}e_{1} dx+\int_{\Omega}u^{p}e_{1} dx.
\endaligned
\end{equation}
Since there exist positive constants $c_{1},c_{2}$ such that $\lambda t^{q}+t^{p}>c_{1}\lambda^{c_{2}}t$, for any $t>0$, we obtain from \eqref{F1}
that $c_{1}\lambda^{c_{2}}<\lambda_{1}$ which implies $\Lambda<\infty$.
\end{proof}

\begin{lem}\label{EAC}
For all $0<\lambda<\Lambda$, the equation \eqref{SCCE2} has a minimal solution.
\end{lem}
\begin{proof}
Let $v_{\lambda}$ be the unique positive solution of
$$
\left\{\begin{array}{l}
\displaystyle-\Delta v=\lambda v^{q}\hspace{5.14mm} \ \mbox{in}\ \ \Omega,\\
\\
\displaystyle v=0 \hspace{12.14mm}\ \ \ \mbox{on}\ \ \partial\Omega.
\end{array}
\right.
$$
We already know that there exists a solution $u>0$ of \eqref{SCCE2} for every $0<\lambda<\Lambda$. Since $-\Delta u\geq\lambda u^{q}$ we can use Lemma 3.3 of \cite{ABC} with $w=u$ to deduce that any solution $u$ of \eqref{SCCE2} must satisfy $u\geq v_{\lambda}$. Clearly, $v_{\lambda}$ is a subsolution of \eqref{SCCE2}. The monotone iteration
$$
-\Delta u_{n+1}=\lambda u_{n}^{q}+u_{n}^{p}+\Big(\int_{\Omega}\frac{
|u_{n}(y)|^{2_{\mu}^{\ast}}}{|x-y|^{\mu}}dy\Big)|u_{n}(x)|^{2_{\mu}^{\ast}-2}u_{n}(x),\ \ \  u_{0}=v_{\lambda},
$$
satisfies $u_{n}\nearrow u_{\lambda}$, with $u_{\lambda}$ solution of \eqref{SCCE2}. It is easy to check that $u_{\lambda}$ is a minimal solution of \eqref{SCCE2}. Indeed, if $u$ is any solution of \eqref{SCCE2}, then $u\geq v_{\lambda}$ and $u$ is a supersolution of \eqref{SCCE2}. Thus $u_{n}\leq u$, $\forall n$, by induction, and $u\geq u_{\lambda}$.

Moreover, this minimal solution is increasing with respect to $\lambda$. In fact, if $u_{\lambda'}$ is a minimal solution of \eqref{SCCE2} with $\lambda=\lambda'$, then we have
$$
-\Delta u_{\lambda'}\leq\lambda'' u_{\lambda'}^{q}+u_{\lambda'}^{p}+\Big(\int_{\Omega}\frac{
|u_{\lambda'}(y)|^{2_{\mu}^{\ast}}}{|x-y|^{\mu}}dy\Big)|u_{\lambda'}(x)|^{2_{\mu}^{\ast}-2}u_{\lambda'}(x)
$$
for $0<\lambda'<\lambda''<\Lambda$, that is $u_{\lambda'}$  is a subsolution of \eqref{SCCE2} with $\lambda=\lambda''$. So $u_{\lambda'}\leq u_{\lambda''}$ and $u_{\lambda'}\not\equiv u_{\lambda''}$ for $0<\lambda'<\lambda''<\Lambda$.
\end{proof}

The solutions for equation \eqref{SCCE2} are classical in fact. First let us recall an important inequality for  nonlocal nonlinearities by Moroz and Van Schaftingen \cite{MS2} which complement  the results by Brezis and Kato in \cite{BK}. .
\begin{lem}\label{EAC3} (See \cite{MS2}.)
Let $N\geq 2$, $\mu\in (0,N)$ and $\theta\in (0,N)$. If $H, K\in L^{\frac{2N}{N-\mu+2}}(\mathbb{R}^N)+L^{\frac{2N}{N-\mu}}(\mathbb{R}^N) $, $(1-\frac{\mu}{N})<\theta<(1+\frac{\mu}{N})$, then for any $\vr>0$, there exists $C_{\varepsilon,\theta}\in\mathbb{R}$ such that for every $u\in H^{1}(\mathbb{R}^N)$,
$$
\int_{\mathbb{R}^N}\Big(|x|^{-\mu}\ast (H|u|^{\theta})\Big)K|u|^{2-\theta}dx\leq \varepsilon^{2}\int_{\mathbb{R}^{N}}|\nabla u|^{2}dx+C_{\varepsilon,\theta}\int_{\mathbb{R}^N}|u|^{2}dx.
$$
\end{lem}

We have the following regularity Lemma and $L^{\infty}$ estimates for the solutions.
\begin{lem}\label{EAC4}
Let $u$ be a solution to the problem
\begin{equation}\label{F2}
\left\{\begin{array}{l}
\displaystyle-\Delta u
=g(x,u)\hspace{4.14mm}\mbox{in}\hspace{1.14mm} \Omega,\\
\displaystyle u\in H_{0}^{1}(\Omega),
\end{array}
\right.
\end{equation}
and assume that $|g(x,u)|\leq C(1+|u|^{p})+\Big(\int_{\Omega}\frac{|u|^{2_{\mu}^{\ast}}}{|x-y|^{\mu}}dy\Big)|u|^{2_{\mu}^{\ast}-2}u$, $1<p<2^{\ast}-1$ and $C>0$, then $u\in L^{\infty}(\Omega)$ and $u\in C^{2}(\overline{\Omega})$.
\end{lem}
\begin{proof}
Let us define the truncation: $\Omega\rightarrow \mathbb{R}$, for $\tau>0$ large,
$$
u_{\tau}(x)=\left\{\begin{array}{l}
\displaystyle -\tau \hspace{10.14mm} \mbox{if}\hspace{2.14mm} u\leq-\tau,\\
\displaystyle u(x) \hspace{8.14mm} \mbox{if} \hspace{2.14mm}-\tau<u<\tau,\\
\displaystyle \tau \hspace{13.14mm} \mbox{if} \hspace{2.14mm}u\geq\tau.\\
\end{array}
\right.
$$
Since $|u_{\tau}|^{s-2}u_{\tau}\in H_{0}^{1}(\Omega)$ for $ s\geq2$, we take $|u_{\tau}|^{s-2}u_{\tau}$ as a test function in \eqref{F2}, we obtain
$$\aligned
\frac{4(s-1)}{s^{2}}\int_{\Omega}&|\nabla(u_{\tau})^{\frac{s}{2}}|^{2}dx
\\
&= (s-1)\int_{\Omega}|u_{\tau}|^{s-2}|\nabla u_{\tau}|^{2}dx\\
&\leq\int_{\Omega}\int_{\Omega}\frac{|u(y)|^{2_{\mu}^{\ast}}}
{|x-y|^{\mu}}dy|u(x)|^{2_{\mu}^{\ast}-1}|u_{\tau}(x)|^{s-1}dx +C\int_{\Omega}|u_{\tau}|^{s-2}u_{\tau}u^{p}dx+C\int_{\Omega}|u_{\tau}|^{s-1}u_{\tau}dx\\
&\leq\int_{\Omega}\int_{\Omega}\frac{|u(y)|^{2_{\mu}^{\ast}}}
{|x-y|^{\mu}}dy|u(x)|^{2_{\mu}^{\ast}-1}|u_{\tau}(x)|^{s-1}dx +2C\int_{\Omega}(1+u^{2})|u_{\tau}|^{s-2}\frac{1+u^{p}}{1+u}dx.
\endaligned
$$
We denote $a(u):=\frac{1+u^{p}}{1+u}$ and have $0\leq a(u)\leq C_{1}(1+u^{p-1})\in L^{\frac{N}{2}}(\Omega)$. If $2\leq s<\frac{2N}{N-\mu}$, using Lemma \ref{EAC3} with $\theta=\frac{2}{s}$, there exists $C_{2}>0$ such that
$$\aligned
\int_{\Omega}\int_{\Omega}\frac{|u_{\tau}(y)|^{2_{\mu}^{\ast}}}{|x-y|^{\mu}}dy
|u_{\tau}(x)|^{2_{\mu}^{\ast}-1}|u_{\tau}(x)|^{s-1}dx
\leq\frac{2(s-1)}{s^{2}}\int_{\Omega}|\nabla(u_{\tau})^{\frac{s}{2}}|^{2}dx
+C_{2}\int_{\Omega}||u_{\tau}|^{\frac{s}{2}}|^{2}dx.
\endaligned
$$
Since $|u_{\tau}|\leq|u|$, we have
$$
\aligned
\frac{2(s-1)}{s^{2}}\int_{\Omega}|\nabla(u_{\tau})^{\frac{s}{2}}|^{2}dx
&\leq C_{2}\int_{\Omega}|u|^{s}dx+\int_{A_{\tau}}\int_{\Omega}\frac{|u(y)|^{2_{\mu}^{\ast}-1}|u(y)|^{s-1}}
{|x-y|^{\mu}}dy|u(x)|^{2_{\mu}^{\ast}}dx\\
&\hspace{1cm}+2C\int_{\Omega}a(u)(1+u^{2})|u_{\tau}|^{s-2}dx,
\endaligned
$$
where $A_{\tau}=\{x\in\Omega:|u|>\tau\}.$

Since $2\leq s<\frac{2N}{N-\mu}$, applying the Hardy-Littlewood-Sobolev inequality again,
$$
\int_{A_{\tau}}\int_{\Omega}\frac{|u(y)|^{2_{\mu}^{\ast}-1}|u(y)|^{s-1}}
{|x-y|^{\mu}}dy|u(x)|^{2_{\mu}^{\ast}}dx\leq C_{3}\Big(\int_{\Omega}||u|^{2_{\mu}^{\ast}-1}|u|^{s-1}|^{r}dx\Big)^{\frac{1}{r}}
\Big(\int_{A_{\tau}}||u|^{2_{\mu}^{\ast}}|^{l}dx\Big)^{\frac{1}{l}},
$$
with $\frac{1}{r}=1+\frac{N-\mu}{2N}-\frac{1}{s}$ and $\frac{1}{l}=\frac{N-\mu}{2N}+\frac{1}{s}$. By H\"{o}lder's inequality, if $u\in L^{s}(\Omega)$, then $|u|^{2_{\mu}^{\ast}}\in L^{l}(\Omega)$ and $|u|^{2_{\mu}^{\ast}-1}|u|^{s-1}\in L^{r}(\Omega)$, whence by Lebesgue's dominated convergence theorem
$$
\lim_{\tau\rightarrow\infty}\int_{A_{\tau}}\int_{\Omega}\frac{|u(y)|^{2_{\mu}^{\ast}-1}|u(y)|^{s-1}}
{|x-y|^{\mu}}dy|u(x)|^{2_{\mu}^{\ast}}dx=0.
$$
On the other hand,
$$\aligned
\int_{\Omega}a(u)u^{2}|u_{\tau}|^{s-2}dx&\leq \tau_{0}\int_{a<\tau_{0}}u^{2}|u_{\tau}|^{s-2}dx +\int_{a\geq\tau_{0}}a(u)u^{2}|u_{\tau}|^{s-2}dx\\
&\leq C_{4}\tau_{0}+(\int_{a\geq\tau_{0}}a(u)^{\frac{N}{2}}dx)^{\frac{2}{N}}(\int_{a\geq\tau_{0}}
(u|u_{\tau}|^{\frac{s-2}{2}})^{2^{\ast}}dx)^{\frac{2}{2^{\ast}}}
\endaligned
$$
and
$$
\int_{\Omega}a(u)|u_{\tau}|^{s-2}dx\leq C_{5}\tau_{0}+(\int_{a\geq\tau_{0}}a(u)^{\frac{N}{2}}dx)^{\frac{2}{N}}(\int_{a\geq\tau_{0}}
(|u_{\tau}|^{\frac{s-2}{2}})^{2^{\ast}}dx)^{\frac{2}{2^{\ast}}},
$$
where, since $|u_{\tau}|^{\frac{s}{2}}$, $C_{4}$ and $C_{5}$ can be taken independent of $\tau$. Hence, by $a(u)\in L^{\frac{N}{2}}$ it follows that
$$
(\int_{a\geq\tau_{0}}a(u)^{\frac{N}{2}}dx)^{\frac{2}{N}}\rightarrow0
$$
as $\tau_{0}\rightarrow\infty$. Therefore, choosing $\tau_{0}$ large enough such that $$
C(\int_{a\geq\tau_{0}}a(u)^{\frac{N}{2}}dx)^{\frac{2}{N}}<\frac{1}{2},$$by Sobolev embedding theory, we obtain that there exists a constant $K(\tau_{0})$, independent of $\tau$, for which it holds
$$
\Big(\int_{\Omega}|u_{\tau}|^{\frac{sN}{N-2}}dx\Big)^{1-\frac{2}{N}}\leq C_{2}\int_{\Omega}|u|^{s}dx+K(\tau_{0}),
$$
Letting $\tau\rightarrow\infty$ we conclude that $u\in L^{\frac{sN}{N-2}}(\Omega)$. By iterating over $s$ a finite number of times we cover the range $s\in[2,\frac{2N}{N-\mu})$. So we can get weak solution $u\in L^{s}(\Omega)$ of \eqref{SCCE2} for every $s\in[2,\frac{2N^{2}}{(N-\mu)(N-2)})$.
Thus, $|u|^{2_{\mu}^{\ast}}\in L^{s}(\Omega)$ for every $s\in[\frac{2(N-2)}{2N-\mu},\frac{2N^{2}}{(N-\mu)(2N-\mu)})$. Since $\frac{2(N-2)}{2N-\mu}<\frac{N}{N-\mu}<\frac{2N^{2}}{(N-\mu)(2N-\mu)}$, we have
$$\int_{\Omega}\frac{|u|^{2_{\mu}^{\ast}}}{|x-y|^{\mu}}dy\in L^{\infty}(\Omega)$$ and so
$$
|g(x,u)|\leq C_{6}(1+|u|^{2^{\ast}-1}).
$$
From Theorem 1.16 of \cite{AM}, we have the weak solution  $u\in L^{\infty}(\Omega)$ and $u\in C^{2}(\overline{\Omega})$.
\end{proof}

\begin{lem}\label{EAC6}
For all $0<\lambda<\Lambda$, there exists a positive solution for \eqref{SCCE2} which is a local minimum of the functional $J_{\lambda}$ in the $C^{1}$-topology.
\end{lem}
\begin{proof} Let $0<\lambda'<\lambda<\lambda''<\Lambda$ and $u_{\lambda'}$ and $u_{\lambda''}$ be the corresponding minimal solutions to \eqref{SCCE2},
$\lambda=\lambda'$ and $\lambda''$ respectively. Denote $u:=u_{\lambda''}-u_{\lambda'}$. Then, since minimal solutions is increasing with respect to $\lambda$, we have
$$
\left\{\begin{array}{l}
\displaystyle-\Delta u\geq0\ \ \mbox{in}\ \ \Omega,\\
\\
\displaystyle u=0\ \ \ \ \ \ \ \mbox{on}\ \  \partial\Omega.
\end{array}
\right.
$$

In order to obtain the existence of positive solution, we may assume that $$f_{\lambda}(u)=\Big(\int_{\Omega}\frac{|u|^{2_{\mu}^{\ast}}}{|x-y|^{\mu}}dy\Big)|u|^{2_{\mu}^{\ast}-2}u+\lambda u^{q}+u^{p},$$ for $u\geq0$ and $f_{\lambda}(u)=0$, for $u<0$. We set
$$
f_{\lambda}^{\ast}(u)=\left\{\begin{array}{l}
\displaystyle f_{\lambda}(u_{\lambda'}) \hspace{9.64mm} \mbox{if}\hspace{2.14mm} u\leq u_{\lambda'},\\
\displaystyle f_{\lambda}(u) \hspace{12.64mm} \mbox{if} \hspace{2.14mm}u_{\lambda'}<u<u_{\lambda''},\\
\displaystyle f_{\lambda}(u_{\lambda''}) \hspace{9.14mm} \mbox{if} \hspace{2.14mm}u\geq u_{\lambda''},\\
\end{array}
\right.
$$
$$
F_{\lambda}^{\ast}(u)=\int_{0}^{u}f_{\lambda}^{\ast}(s)ds
$$
and
$$
J_{\lambda}^{\ast}(u)=\frac{1}{2}\int_{\Omega}|\nabla u|^{2}dx-\int_{\Omega}
F_{\lambda}^{\ast}(u)dx.
$$
Standard calculation shows that $J_{\lambda}^{\ast}$ achieves its global minimum at some $u_{0}\in H_{0}^{1}(\Omega)$ satisfying
$$
\left\{\begin{array}{l}
\displaystyle-\Delta u_{0}=f_{\lambda}^{\ast}(u_{0})\ \ \mbox{in}\ \ \Omega,\\
\\
\displaystyle u_{0}=0\ \ \ \ \ \ \ \ \ \ \ \ \ \mbox{on}\ \  \partial\Omega
\end{array}
\right.
$$
and
$$
J_{\lambda}^{\ast}(u_{0})\leq J_{\lambda}^{\ast}(u), \ \  \ \forall u\in H_{0}^{1}(\Omega).
$$
By the Maximum Principle, we get that $u_{\lambda'}<u_{0}<u_{\lambda''}$ in $\Omega$, as well as
$$
\frac{\partial}{\partial\nu}(u_{0}-u_{\lambda'})<0,\ \ \ \frac{\partial}{\partial\nu}(u_{0}-u_{\lambda''})>0,\ \ \ x\in\partial\Omega.
$$
where $\nu$ is the outer unit normal at $\partial\Omega$. If $\|u-u_{0}\|_{C_{0}^{1}(\Omega)}\leq \varepsilon $ for $\vr$ small enough,  then $u_{\lambda'}\leq u\leq u_{\lambda''}$ and so we have $J_{\lambda}^{\ast}(u)=J_{\lambda}(u)$. Then
$$
J_{\lambda}(u)=J_{\lambda}^{\ast}(u)\geq J_{\lambda}^{\ast}(u_{0})=J_{\lambda}(u_{0})
$$
for any $u\in C_{0}^{1}(\Omega)$ with $
\|u-u_{0}\|_{C_{0}^{1}(\Omega)}\leq \varepsilon$ and so $u_0$ is a local minimum for $J_{\lambda}$ in the sense of $C^1$ topology.
\end{proof}

\begin{lem}\label{EAC7}
Let $u_{0}\in H_{0}^{1}(\Omega)$ be a local minimum of the functional $J_{\lambda}$ in $ C_{0}^{1}(\Omega)$, by this we mean that there exists $r>0$ such that
$$
J_{\lambda}(u_{0})\leq J_{\lambda}(u_{0}+u),\ \ \forall u\in C_{0}^{1}(\Omega)\ \ \mbox{with}\ \ \|u\|_{C_{0}^{1}(\Omega)}\leq r.
$$
Then $u_{0}$ is a local minimum of $J_{\lambda}$ in $ H_{0}^{1}(\Omega)$, that is, there exists $\varepsilon_{0}>0$ such that
$$
J_{\lambda}(u_{0})\leq J_{\lambda}(u_{0}+u),\ \ \forall u\in H_{0}^{1}(\Omega)\ \ \mbox{with}\ \ \|u\|\leq \varepsilon_{0}.
$$
\end{lem}
\begin{proof}
Arguing by contradiction,  we may suppose that for any $\varepsilon>0$ small we have
$$
\min_{u\in B_{\varepsilon}(u_{0})}J_{\lambda}(u)<J_{\lambda}(u_{0}),
$$
where $B_{\varepsilon}(u_{0})=\{u\in H_{0}^{1}(\Omega):\|u-u_{0}\|\leq \varepsilon\}$.

Following Brezis and Nirenberg \cite{BN1}, we sketch the proof here. Applying a standard argument of weak lower semi-continuity, we may take $u_{0,\varepsilon}$ such that $\min_{u\in B_{\varepsilon}(u_{0})}J_{\lambda}(u)=J_{\lambda}(u_{0,\varepsilon})$. We need to prove  that
$u_{0,\varepsilon}\rightarrow u_{0}$ in $C_{0}^{1}(\Omega)$ as $\varepsilon\searrow0$. Note that the Euler-Lagrange equation satisfied by $u_{0,\varepsilon}$ involves a Lagrange multiplier $\zeta_{\varepsilon}$ such that
\begin{equation}\label{F3}
\langle J_{\lambda}'(u_{0,\varepsilon}),\varphi\rangle=\zeta_{\varepsilon}\langle u_{0,\varepsilon},\varphi\rangle_{H_{0}^{1}(\Omega)}, \ \ \ \ \forall\varphi\in H_{0}^{1}(\Omega).
\end{equation}
From $u_{0,\varepsilon}$ is a minimum of $J_{\lambda}$ in $B_{\varepsilon}(u_{0})$, we have
\begin{equation}\label{F4}
\zeta_{\varepsilon}=\frac{\langle J_{\lambda}'(u_{0,\varepsilon}),u_{0,\varepsilon}\rangle}{\|u_{0,\varepsilon}\|^{2}}\leq0.
\end{equation}
By \eqref{F3} we easily get that $u_{0,\varepsilon}$ satisfies
$$
\left\{\begin{array}{l}
\displaystyle-\Delta u_{0,\varepsilon}=\frac{1}{1-\zeta_{\varepsilon}}f_{\lambda}(u_{0,\varepsilon})
:=f_{\lambda,\varepsilon}(u_{0,\varepsilon})\ \ \ \ \ \mbox{in}\ \ \Omega,\\
\\
\displaystyle u_{0,\varepsilon}=0\ \ \hspace{49.14mm} \mbox{on}\ \  \partial\Omega.
\end{array}
\right.
$$
Since $u_{0,\varepsilon}\in H_{0}^{1}(\Omega)$ and \eqref{F4},
by Lemma \ref{EAC4}, there exists a constant $C$ independent of $\varepsilon$ such that $\|u_{0,\varepsilon}\|_{C^{2}(\overline{\Omega})}<C$. By Ascoli-Arzel\'{a} Theorem there exists
a subsequence, still denoted by $u_{0,\varepsilon}$, such that $u_{0,\varepsilon}\rightarrow u_{0}$ uniformly in $C_{0}^{1}(\Omega)$ as $\varepsilon\searrow0$. This implies that
for $\varepsilon$ small enough,
$$
J_{\lambda}(u_{0,\varepsilon})<J_{\lambda}(u_{0})
$$
for any $u_{0,\varepsilon}$ with $\|u_{0,\varepsilon}-u_{0}\|_{C_{0}^{1}(\Omega)}< \varepsilon.$
This contradicts our hypothesis.
\end{proof}

From Lemmas \ref{EAC6} and \ref{EAC7}, we know there exists a local minimum $u_{0}$ in $H_{0}^{1}(\Omega)$. For $0<\lambda<\Lambda$, we consider the translated nonlinearity defined by
$$
g(u)=\left\{\begin{array}{ll}
\displaystyle  \mathbb{F}(u)\hspace{2.14mm} \mbox{if}\hspace{2.14mm} u\geq0,\\
\\
\displaystyle 0 \hspace{8.14mm} \mbox{if} \hspace{2.14mm}u<0,\\
\end{array}
\right.
$$
where
$$
\aligned
\mathbb{F}(u)&=\lambda (u_{0}+u)^{q}-\lambda u_{0}^{q}+(u_{0}+u)^{p}- u_{0}^{p}\\
&\hspace{1cm}+\Big(\int_{\Omega}\frac{|u_{0}+u|^{2_{\mu}^{\ast}}}{|x-y|^{\mu}}dy\Big)|u_{0}+u|^{2_{\mu}^{\ast}-2}(u_{0}+u)
-\displaystyle\Big(\int_{\Omega}\frac{|u_{0}|^{2_{\mu}^{\ast}}}{|x-y|^{\mu}}dy\Big)|u_{0}|^{2_{\mu}^{\ast}-2}u_{0}.
\endaligned
$$
We consider the translated problem
\begin{equation}\label{F5}
\left\{\begin{array}{l}
\displaystyle-\Delta u=g(u)\ \ \ \ \mbox{in}\ \ \Omega,\\
\\
\displaystyle u=0\ \ \ \ \ \ \ \ \ \ \ \ \ \mbox{on}\ \  \partial\Omega.
\end{array}
\right.
\end{equation}
The Hardy-Littlewood-Sobolev inequality implies that
$$
\bar{J}_{\lambda}(u)=\frac{1}{2}\int_{\Omega}|\nabla u|^{2}dx-\int_{\Omega}
G(u)dx,
$$
is well defined, where
$$
G(u)=\int_{0}^{u}g(s)ds.
$$
Therefore  if $\bar{u}\not\equiv0$ is a critical point of $\bar{J}_{\lambda}$ then it is a solution
of \eqref{F5}, by the Maximum Principle, $\bar{u}>0$ and  $u=u_{0}+\bar{u}$ will
be a second solution of \eqref{SCCE2}. In order to finish the proof the Theorem \ref{EXS4}, we are going to investigate the existence of nontrivial critical points for $\bar{J}_{\lambda}$.

\begin{lem}\label{EAC8}
$u=0$ is a local minimum of $\bar{J}_{\lambda}$ in $ H_{0}^{1}(\Omega)$.
\end{lem}
\begin{proof}
We only need to show that $u=0$ is a local minimum of $\bar{J}_{\lambda}$ in $C^1$ topology. Let $u\in C_{0}^{1}(\Omega)$, by direct computation, we know
$$\aligned
\bar{J}_{\lambda}(u)&=\frac{1}{2}\|u\|^{2}-\int_{\Omega}
G(u)dx\\
&=\frac{1}{2}\|u\|^{2}-\frac{\lambda}{q+1} \int_{\Omega}|u_{0}+u|^{q+1}dx+\frac{\lambda}{q+1} \int_{\Omega}|u_{0}|^{q+1}dx+\lambda \int_{\Omega}u_{0}^{q}udx-\frac{1}{p+1} \int_{\Omega}|u_{0}+u|^{p+1}dx\\
&\hspace{0.5cm}+\frac{1}{p+1} \int_{\Omega}|u_{0}|^{p+1}dx+ \int_{\Omega}u_{0}^{p}udx-\frac{1}{2\cdot2_{\mu}^{\ast}}\int_{\Omega}\int_{\Omega}
\frac{|(u_{0}+u)(x)|^{2_{\mu}^{\ast}}|(u_{0}+u)(y)|^{2_{\mu}^{\ast}}}{|x-y|^{\mu}}dxdy\\
&\hspace{1cm}
+\frac{1}{2\cdot2_{\mu}^{\ast}}\int_{\Omega}\int_{\Omega}
\frac{|u_{0}(x)|^{2_{\mu}^{\ast}}|u_{0}(y)|^{2_{\mu}^{\ast}}}{|x-y|^{\mu}}dxdy
+\int_{\Omega}\Big(\int_{\Omega}
\frac{|u_{0}|^{2_{\mu}^{\ast}}}{|x-y|^{\mu}}dy\Big)|u_{0}|^{2_{\mu}^{\ast}-2}u_{0}udx.
\endaligned
$$

On the other hand,
$$\aligned
J_{\lambda}&(u_{0}+u)\\
&=\frac{1}{2}\|u_{0}+u\|^{2}-\frac{\lambda}{q+1} \int_{\Omega}|u_{0}+u|^{q+1}dx-\frac{1}{p+1} \int_{\Omega}|u_{0}+u|^{p+1}dx\\
&\hspace{7mm}-\frac{1}{2\cdot2_{\mu}^{\ast}}\int_{\Omega}\int_{\Omega}
\frac{|(u_{0}+u)(x)|^{2_{\mu}^{\ast}}|(u_{0}+u)(y)|^{2_{\mu}^{\ast}}}{|x-y|^{\mu}}dxdy\\
&=\frac{1}{2}\|u_{0}\|^{2}+\frac{1}{2}\|u\|^{2}+\int_{\Omega}\nabla u\nabla u_{0}dx -\frac{\lambda}{q+1} \int_{\Omega}|u_{0}+u|^{q+1}dx-\frac{1}{p+1} \int_{\Omega}|u_{0}+u|^{p+1}dx\\
&\hspace{7mm}-\frac{1}{2\cdot2_{\mu}^{\ast}}\int_{\Omega}\int_{\Omega}
\frac{|(u_{0}+u)(x)|^{2_{\mu}^{\ast}}|(u_{0}+u)(y)|^{2_{\mu}^{\ast}}}{|x-y|^{\mu}}dxdy\\
&=\frac{1}{2}\|u_{0}\|^{2}+\frac{1}{2}\|u\|^{2} -\frac{\lambda}{q+1} \int_{\Omega}|u_{0}+u|^{q+1}dx-\frac{1}{p+1} \int_{\Omega}|u_{0}+u|^{p+1}dx+\lambda \int_{\Omega}u_{0}^{q}udx+\int_{\Omega}u_{0}^{p}udx\\
&\hspace{7mm}-\frac{1}{2\cdot2_{\mu}^{\ast}}\int_{\Omega}\int_{\Omega}
\frac{|(u_{0}+u)(x)|^{2_{\mu}^{\ast}}|(u_{0}+u)(y)|^{2_{\mu}^{\ast}}}{|x-y|^{\mu}}dxdy
+\int_{\Omega}\Big(\int_{\Omega}
\frac{|u_{0}|^{2_{\mu}^{\ast}}}{|x-y|^{\mu}}dy\Big)|u_{0}|^{2_{\mu}^{\ast}-2}u_{0}udx.
\endaligned
$$
Since $u_{0}$ is a local minimum of $J_{\lambda}$, we have that
$$\aligned
\bar{J}_{\lambda}(u)&=J_{\lambda}(u_{0}+u)-\frac{1}{2}\|u_{0}\|^{2}+\frac{\lambda}{q+1} \int_{\Omega}|u_{0}|^{q+1}dx\\
&\hspace{7mm}+\frac{1}{p+1} \int_{\Omega}|u_{0}|^{p+1}dx
+\frac{1}{2\cdot2_{\mu}^{\ast}}\int_{\Omega}\int_{\Omega}
\frac{|u_{0}(x)|^{2_{\mu}^{\ast}}|u_{0}(y)|^{2_{\mu}^{\ast}}}{|x-y|^{\mu}}dxdy\\
&=J_{\lambda}(u_{0}+u)-J_{\lambda}(u_{0})\\
&\geq0\\
&=\bar{J}_{\lambda}(0)
\endaligned
$$
provided $\|u\|_{C_{0}^{1}(\Omega)}<\varepsilon$.
\end{proof}

\begin{lem}\label{EAC9}
If $u=0$ is the only critical point of $\bar{J}_{\lambda}$ in $H_{0}^{1}(\Omega)$ then $\bar{J}_{\lambda}$ satisfies a local Palais-Smale condition
below the critical level $\frac{N+2-\mu}{4N-2\mu}S_{H,L}^{\frac{2N-\mu}{N+2-\mu}}$.

\end{lem}
\begin{proof}
If $\{w_{n}\}$ is a $(PS)_c$ sequence of $\bar{J}_{\lambda}$, then
\begin{equation}\label{F6}
\bar{J}_{\lambda}(w_{n})\rightarrow c<\frac{N+2-\mu}{4N-2\mu}S_{H,L}^{\frac{2N-\mu}{N+2-\mu}}, \ \ \ \bar{J}_{\lambda}'(w_{n})\rightarrow0.
\end{equation}
Since the fact that $w_{0}$ is a critical point implies $\bar{J}_{\lambda}(w_{n})=J_{\lambda}(u_{n})-J_{\lambda}(w_{0})$, where $u_{n}=w_{n}+w_{0}$, we
have that
\begin{equation}\label{F7}
J_{\lambda}(u_{n})\rightarrow c+J_{\lambda}(w_{0}), \ \ \ J_{\lambda}'(u_{n})\rightarrow0.
\end{equation}

Similar to Lemma \ref{WSo}, we have, if $\{u_{n}\}$ is a $(PS)_c$ sequence of $J_{\lambda}$, then $\{u_{n}\}$ is bounded, if $u_{\infty}\in H_{0}^{1}(\Omega)$ is the weak limit of $\{u_{n}\}$, then $u_{\infty}$ is a weak solution of problem \eqref{SCCE2}. Let $v_{n}:=u_{n}-u_{\infty}$, then we know $v_{n}\rightharpoonup0$ in $H_{0}^{1}(\Omega)$ and $v_{n}\rightarrow 0$ a.e. in $\Omega$.
From by the proof of Lemma \ref{ConPro}, we can assume there exists a nonnegative constant $b$ such that
$$
\int_{\Omega}|\nabla v_{n}|^{2}dx\rightarrow b
$$
as $n\rightarrow+\infty$ and we obtain
\begin{equation}\label{F8}
\lim_{n\rightarrow\infty}J_{\lambda}(u_{n})\geq \frac{N+2-\mu}{4N-2\mu}S_{H,L}^{\frac{2N-\mu}{N+2-\mu}}+J_{\lambda}(u_{\infty})
\end{equation}
or $b=0$. If \eqref{F8} holds, then by \eqref{F6} and \eqref{F7}, we have
$$
\frac{N+2-\mu}{4N-2\mu}S_{H,L}^{\frac{2N-\mu}{N+2-\mu}}+J_{\lambda}(u_{\infty})< \frac{N+2-\mu}{4N-2\mu}S_{H,L}^{\frac{2N-\mu}{N+2-\mu}}+J_{\lambda}(w_{0}).
$$
Note that $u=0$ is the only critical point of $\bar{J}_{\lambda}$ in $H_{0}^{1}(\Omega)$ then, $u_{\infty}=w_{0}$. Thus we have
$$
\frac{N+2-\mu}{4N-2\mu}S_{H,L}^{\frac{2N-\mu}{N+2-\mu}}< \frac{N+2-\mu}{4N-2\mu}S_{H,L}^{\frac{2N-\mu}{N+2-\mu}}.
$$
This a contradiction. Then $b=0$, this is
$$
\|u_{n}-u_{\infty}\|\rightarrow0
$$
as $n\rightarrow+\infty$. This ends the proof of Lemma \ref{EAC9}.
\end{proof}

\begin{lem}\label{EAC10} Let $0<q<1$, $1<p<2^{\ast}-1$, $\lambda>0$ and $u_{\varepsilon}$ as defined in \eqref{EFV}. Then, there exists $\varepsilon>0$ small enough such that
\begin{equation}\label{F11}
\sup_{t\geq0}\bar{J}_{\lambda}(tu_{\varepsilon})<\frac{N+2-\mu}{4N-2\mu}S_{H,L}^{\frac{2N-\mu}{N+2-\mu}}.
\end{equation}
\end{lem}
\begin{proof}  Since
$$
(a+b)^{p}\geq a^{p}+b^{p}+pa^{p-1}b
$$
for every $a,b\geq0$ and $p>1$, then

$$\aligned
&\hspace{0.5cm}\int_{\Omega}\int_{\Omega}\frac{|(u_{\varepsilon}+u_{0})(x)|^{2_{\mu}^{\ast}}|(u_{\varepsilon}+u_{0})(y)|^{2_{\mu}^{\ast}}}
{|x-y|^{\mu}}dxdy\\
&\geq\int_{\Omega}\int_{\Omega}\frac{|u_{\varepsilon}(x)|^{2_{\mu}^{\ast}}|u_{\varepsilon}(y)|^{2_{\mu}^{\ast}}}
{|x-y|^{\mu}}dxdy+\int_{\Omega}\int_{\Omega}\frac{|u_{0}(x)|^{2_{\mu}^{\ast}}|u_{0}(y)|^{2_{\mu}^{\ast}}}
{|x-y|^{\mu}}dxdy+2_{\mu}^{\ast}\int_{\Omega}\int_{\Omega}\frac{|u_{\varepsilon}(x)|^{2_{\mu}^{\ast}}|u_{0}(y)||u_{\varepsilon}(y)|^{2_{\mu}^{\ast}-1}}
{|x-y|^{\mu}}dxdy\\
&\hspace{0.5cm}+2_{\mu}^{\ast}\int_{\Omega}\int_{\Omega}\frac{|u_{0}(y)|^{2_{\mu}^{\ast}}|u_{0}(x)|^{2_{\mu}^{\ast}-1}|u_{\varepsilon}(x)|}
{|x-y|^{\mu}}dxdy+2_{\mu}^{\ast}\int_{\Omega}\int_{\Omega}\frac{|u_{0}(x)|^{2_{\mu}^{\ast}}|u_{0}(y)|^{2_{\mu}^{\ast}-1}|u_{\varepsilon}(y)|}
{|x-y|^{\mu}}dxdy\\
&\geq\int_{\Omega}\int_{\Omega}\frac{|u_{\varepsilon}(x)|^{2_{\mu}^{\ast}}|u_{\varepsilon}(y)|^{2_{\mu}^{\ast}}}
{|x-y|^{\mu}}dxdy+\int_{\Omega}\int_{\Omega}\frac{|u_{0}(x)|^{2_{\mu}^{\ast}}|u_{0}(y)|^{2_{\mu}^{\ast}}}
{|x-y|^{\mu}}dxdy\\
&\hspace{0.5cm}+2_{\mu}^{\ast}C_{1}\int_{B_{\delta}}\int_{B_{\delta}}\frac{|u_{\varepsilon}(x)|^{2_{\mu}^{\ast}}|u_{\varepsilon}(y)|^{2_{\mu}^{\ast}-1}}
{|x-y|^{\mu}}dxdy+2\cdot2_{\mu}^{\ast}\int_{\Omega}\int_{\Omega}\frac{|u_{0}(y)|^{2_{\mu}^{\ast}}|u_{0}(x)|^{2_{\mu}^{\ast}-1}|u_{\varepsilon}(x)|}
{|x-y|^{\mu}}dxdy,
\endaligned
$$
thanks to $u_{0}\geq C_{1}>0$ on $B_{\delta}$.
On the other hand, since for every $u\geq0$
$$
g(u)\geq\Big(\int_{\Omega}\frac{|u_{0}+u|^{2_{\mu}^{\ast}}}{|x-y|^{\mu}}dy\Big)
|u_{0}+u|^{2_{\mu}^{\ast}-2}(u_{0}+u)-\Big(\int_{\Omega}\frac{|u_{0}|^{2_{\mu}^{\ast}}}
{|x-y|^{\mu}}dy\Big)|u_{0}|^{2_{\mu}^{\ast}-2}u_{0},
$$
we have
$$\aligned
\bar{J}_{\lambda}(u_{\varepsilon})&\leq \frac{1}{2}\int_{\Omega}|\nabla u_{\varepsilon}|^{2}dx-\frac{1}{2\cdot2_{\mu}^{\ast}}
\int_{\Omega}\int_{\Omega}\frac{|(u_{\varepsilon}+u_{0})(x)|^{2_{\mu}^{\ast}}|(u_{\varepsilon}+u_{0})(y)|^{2_{\mu}^{\ast}}}
{|x-y|^{\mu}}dxdy\\
&\hspace{10.14mm}+\frac{1}{2\cdot2_{\mu}^{\ast}}
\int_{\Omega}\int_{\Omega}\frac{|u_{0}(x)|^{2_{\mu}^{\ast}}|u_{0}(y)|^{2_{\mu}^{\ast}}}
{|x-y|^{\mu}}dxdy+\int_{\Omega}\Big(\int_{\Omega}
\frac{|u_{0}|^{2_{\mu}^{\ast}}}{|x-y|^{\mu}}dy\Big)|u_{0}|^{2_{\mu}^{\ast}-2}u_{0}u_{\varepsilon}dx\\
&\leq\frac{1}{2}\int_{\Omega}|\nabla u_{\varepsilon}|^{2}dx-\frac{1}{2\cdot2_{\mu}^{\ast}}
\int_{\Omega}\int_{\Omega}\frac{|u_{\varepsilon}(x)|^{2_{\mu}^{\ast}}|u_{\varepsilon}(y)|^{2_{\mu}^{\ast}}}
{|x-y|^{\mu}}dxdy-C_{2}\int_{B_{\delta}}\int_{B_{\delta}}\frac{|u_{\varepsilon}(x)|^{2_{\mu}^{\ast}}
|u_{\varepsilon}(y)|^{2_{\mu}^{\ast}-1}}
{|x-y|^{\mu}}dxdy.
\endaligned
$$
By a direct computation, we know
\begin{equation}\label{F16}
\aligned
\int_{B_{\delta}}\int_{B_{\delta}}&\frac{|u_{\varepsilon}(x)|^{2_{\mu}^{\ast}}
|u_{\varepsilon}(y)|^{2_{\mu}^{\ast}-1}}
{|x-y|^{\mu}}dxdy\\
&=\int_{B_{\delta}}\int_{B_{\delta}}\frac{|U_{\varepsilon}(x)|^{2_{\mu}^{\ast}}
|U_{\varepsilon}(y)|^{2_{\mu}^{\ast}-1}}{|x-y|^{\mu}}dxdy\\
&=\varepsilon^{\frac{2\mu-3N-2}{2}}[N(N-2)]^{\frac{3N-2\mu+2}{4}}
\int_{B_{\delta}}\int_{B_{\delta}}\frac{1}
{(1+|\frac{x}{\varepsilon}|^{2})^{\frac{2N-\mu}{2}}|x-y|^{\mu}(1+|\frac{y}
{\varepsilon}|^{2})^{\frac{N-\mu+2}{2}}}dxdy\\
&=\varepsilon^{\frac{2\mu-3N-2}{2}}[N(N-2)]^{\frac{3N-2\mu+2}{4}}\varepsilon^{2N-\mu}
\int_{B_{\frac{\delta}{\varepsilon}}}\int_{B_{\frac{\delta}{\varepsilon}}}\frac{1}
{(1+|x|^{2})^{\frac{2N-\mu}{2}}|x-y|^{\mu}
(1+|y|^{2})^{\frac{N-\mu+2}{2}}}dxdy\\
&\geq O(\varepsilon^{\frac{N-2}{2}})\int_{B_{\delta}}\int_{B_{\delta}}\frac{1}
{(1+|x|^{2})^{\frac{2N-\mu}{2}}|x-y|^{\mu}
(1+|y|^{2})^{\frac{N-\mu+2}{2}}}dxdy\\
&=O(\varepsilon^{\frac{N-2}{2}})
\endaligned
\end{equation}
provided $\varepsilon<1$. Therefore, by \eqref{C9}, \eqref{C10} and \eqref{F16}, we have
$$\aligned
\bar{J}_{\lambda}(tu_{\varepsilon})&\leq\frac{t^{2}}{2}\int_{\Omega}|\nabla u_{\varepsilon}|^{2}dx-\frac{t^{2\cdot2_{\mu}^{\ast}}}{2\cdot2_{\mu}^{\ast}}
\int_{\Omega}\int_{\Omega}\frac{|u_{\varepsilon}(x)|^{2_{\mu}^{\ast}}|u_{\varepsilon}(y)|^{2_{\mu}^{\ast}}}
{|x-y|^{\mu}}dxdy-C_{2}t^{2\cdot2_{\mu}^{\ast}-1}\int_{\Omega}\int_{\Omega}\frac{|u_{\varepsilon}(x)|^{2_{\mu}^{\ast}}
|u_{\varepsilon}(y)|^{2_{\mu}^{\ast}-1}}
{|x-y|^{\mu}}dxdy\\
&\leq\frac{t^{2}}{2}(C(N,\mu)^{\frac{N-2}{2N-\mu}\cdot\frac{N}{2}}S_{H,L}^{\frac{N}{2}}+O(\varepsilon^{N-2}))
-\frac{t^{2\cdot2_{\mu}^{\ast}}}{2\cdot2_{\mu}^{\ast}}(C(N,\mu)^{\frac{N}{2}}S_{H,L}^{\frac{2N-\mu}{2}}-O(\varepsilon^{N-\frac{\mu}{2}}))
-t^{2\cdot2_{\mu}^{\ast}-1}O(\varepsilon^{\frac{N-2}{2}})\\
&:=g(t).
\endaligned
$$
It is clear that $g(t)\rightarrow -\infty$ as $t\rightarrow+\infty$. It follows that there exists $t_{\varepsilon}>0$ such that $\sup_{t>0}g(t)$ is attained at $t_{\varepsilon}$. Differentiating $g(t)$ and equaling to zero, we obtain that
$$
t_{\varepsilon}(C(N,\mu)^{\frac{N-2}{2N-\mu}\cdot\frac{N}{2}}S_{H,L}^{\frac{N}{2}}+O(\varepsilon^{N-2}))
-t_{\varepsilon}^{2\cdot2_{\mu}^{\ast}-1}(C(N,\mu)^{\frac{N}{2}}S_{H,L}^{\frac{2N-\mu}{2}}-O(\varepsilon^{N-\frac{\mu}{2}}))
-t_{\varepsilon}^{2\cdot2_{\mu}^{\ast}-2}O(\varepsilon^{\frac{N-2}{2}})=0
$$
and so
$$
t_{\varepsilon}<\Big(\frac{C(N,\mu)^{\frac{N-2}{2N-\mu}\cdot\frac{N}{2}}S_{H,L}^{\frac{N}{2}}+O(\varepsilon^{N-2})}
{C(N,\mu)^{\frac{N}{2}}S_{H,L}^{\frac{2N-\mu}{2}}-O(\varepsilon^{N-\frac{\mu}{2}})}\Big)^{\frac{1}{22_{\mu}^{\ast}-2}}:=S_{H,L}(\varepsilon)
$$
and there exists $t_{0}>0$ such that for $\varepsilon>0$ small enough
$$
t_{\varepsilon}>t_{0}.
$$
Since the function
$$
t\mapsto \frac{t^{2}}{2}(C(N,\mu)^{\frac{N-2}{2N-\mu}\cdot\frac{N}{2}}S_{H,L}^{\frac{N}{2}}+O(\varepsilon^{N-2}))
-\frac{t^{2\cdot2_{\mu}^{\ast}}}{2\cdot2_{\mu}^{\ast}}(C(N,\mu)^{\frac{N}{2}}S_{H,L}^{\frac{2N-\mu}{2}}-O(\varepsilon^{N-\frac{\mu}{2}}))
$$
is increasing on $[0,S_{H,L}(\varepsilon)]$, we have
$$\aligned
\max_{t\geq0}\bar{J}_{\lambda}(tu_{\varepsilon})
&\leq\frac{N+2-\mu}{4N-2\mu}\Big(\frac{C(N,\mu)^{\frac{N-2}{2N-\mu}\cdot\frac{N}{2}}S_{H,L}^{\frac{N}{2}}+O(\varepsilon^{N-2})}
{\Big(C(N,\mu)^{\frac{N}{2}}S_{H,L}^{\frac{2N-\mu}{2}}-O(\varepsilon^{N-\frac{\mu}{2}})\Big)^{\frac{N-2}{2N-\mu}}}\Big)^{\frac{2N-\mu}{N+2-\mu}}
-O(\varepsilon^{\frac{N-2}{2}})\\
&\leq\frac{N+2-\mu}{4N-2\mu}S_{H,L}^{\frac{2N-\mu}{N+2-\mu}}+O(\varepsilon^{\min\{N-2,N-\frac{\mu}{2}\}})
-O(\varepsilon^{\frac{N-2}{2}})\\
&<\frac{N+2-\mu}{4N-2\mu}S_{H,L}^{\frac{2N-\mu}{N+2-\mu}},\\
\endaligned
$$
thanks to $t_{0}<t_{\varepsilon}<S_{H,L}(\varepsilon)$ and \eqref{F16}.
\end{proof}

\noindent
{\bf Proof of Theorem \ref{EXS4}.} For every $u_{1}\in H_{0}^{1}(\Omega)\backslash\ \{0\}$, one easily checks that
$$
\bar{J}_{\lambda}(tu_{1})<0
$$
for $t>0$ large enough. Combining with Lemma \ref{EAC8}, we have $\bar{J}_{\lambda}$ has the Mountain pass geometry. Then there
exists a $(PS)$ sequence $\{u_{n}\}$ such that $\bar{J}_{\lambda}(u_{n})\rightarrow c$ and $ \bar{J}_{\lambda}'(u_{n})\rightarrow0$ in $H_{0}^{1}(\Omega)^{-1}$ at the minimax level
$$
c^*=\inf\limits_{\gamma\in\Gamma}\max\limits_{t\in[0,1]}\bar{J}_{\lambda}(\gamma(t))>0,
$$
where
$$
\Gamma:=\{\gamma\in C([0,1],H_{0}^{1}(\Omega)):\gamma(0)=0,\bar{J}_{\lambda}(\gamma(1))<0\}.
$$

From Lemma \ref{EAC10}, we know there exists $\varepsilon>0$ small enough such that
$$
\sup_{t\geq0}\bar{J}_{\lambda}(tu_{\varepsilon})<\frac{N+2-\mu}{4N-2\mu}S_{H,L}^{\frac{2N-\mu}{N+2-\mu}}.
$$
Therefore, by the definition of $c^*$, we know $c^*<\frac{N+2-\mu}{4N-2\mu}S_{H,L}^{\frac{2N-\mu}{N+2-\mu}}$. This estimate jointly with Lemma \ref{EAC9} and the Mountain Pass Theorem if the minimax energy level is positive, or
the refinement of the Mountain Pass Theorem \cite{GP} if the minimax level is zero, gives the existence of a second solution to \eqref{SCCE2}.  $\hfill{} \Box$

\section{\large Infinitely many solutions in the case of a sublinear perturbation}
In this Section, we will study the existence of infinitely many solutions for the critical Choquard equation with  sublinear local perturbation, i.e. the equation \eqref{CCE1} with the exponent of the perturbation satisfying  $0<q<1$. By applying the Dual Fountain Theorem in \cite{BW}, we are going to prove that the energy functional $J_{\lambda}$ has infinitely many critical values.

We denote the sequence of eigenvalues of the operator $-\Delta$ on $\Omega$ with homogeneous Dirichlet boundary data by
$$
0<\lambda_{1}<\lambda_{2}\leq...\leq \lambda_{j}\leq\lambda_{j+1}\leq...
$$
Moreover, $\{e_{j}\}_{j\in\mathbb{N}}\subset L^{\infty}(\Omega)$ will be the sequence of eigenfunctions corresponding to $\lambda_{j}$ which is also an orthogonal basis of $H_{0}^{1}(\Omega)$. Define $X_{j}:=\R e_{j}$, we will use the following notations:
$$
Y_{k}:=\oplus_{j=0}^{k}X_{j}, Z_{k}:=\overline{\oplus_{j=k}^{\infty}X_{j}},
$$
$$
B_{k}:=\{u\in Y_{k}: \|u\|\leq\rho_{k}\}, N_{k}:=\{u\in Z_{k}: \|u\|\leq r_{k}\}
$$
where $\rho_{k}>r_{k}>0$.

\begin{Def}\label{MA2} (See \cite{Wi}.)  Let $X$ be a Banach space, $I\in C^1(X,\R)$ and $c\in \R$. The function $I$ satisfies the $(PS)_c^{\ast}$ condition (with respect to $(Y_{n})$) if any sequence $\{u_{n}\}\subset X$ such that
$$
u_{n_{j}}\in Y_{n_{j}}, I(u_{n_{j}})\rightarrow c,\ \hbox{and}\ \ \  I|_{Y_{n_{j}}}^{'}(u_{n_{j}})\rightarrow 0,\ \ \  \hbox{as}\ \ n_{j}\rightarrow+\infty
$$
contains a subsequence converging to a critical point of $I$.
\end{Def}

\begin{thm}\label{MA3}
[Dual Fountain Theorem] (See \cite{BW}.)
Let $X$ be a Banach space, $I\in C^1(X,\R)$ is an even functional. If, for every $k\geq k_{0}\geq2$, there exists $\rho_{k}>r_{k}>0$ such that \\
$(B_{1})$ $a_{k}:=\inf_{u\in Z_{k},\|u\|=\rho_{k}}I(u)\geq0$,\\
$(B_{2})$ $b_{k}:=\max_{u\in Y_{k},\|u\|=r_{k}}I(u)<0$,\\
$(B_{3})$ $d_{k}:=\inf_{u\in Z_{k},\|u\|\leq\rho_{k}}I(u)\rightarrow0$, $k\rightarrow\infty$,\\
$(B_{4})$ $I$ satisfies the $(PS)_c^{\ast}$ condition for every $[d_{k_{0}},0)$,\\
then $I$ has a sequence of negative critical values converging to 0.
\end{thm}

We are ready to establish the following convergence criteria for the $ (PS) _c$ sequences.
\begin{lem}\label{MA4} Let $0<q<1$. If there exists $M_0>0$ such that, for any $\lambda>0$ and
\begin{equation}\label{E1}
c<\frac{N+2-\mu}{4N-2\mu}S_{H,L}^{\frac{2N-\mu}{N+2-\mu}}-
M_0\lambda^{\frac{2\cdot2_{\mu}^{\ast}}{2\cdot2_{\mu}^{\ast}-q-1}},
\end{equation}
then $J_{\lambda}$ satisfies the $(PS)_c^{\ast}$ condition.
\end{lem}
\begin{proof}
Consider a sequence $\{u_{n_{j}}\}\subset H_{0}^{1}(\Omega)$ such that
$$
u_{n_{j}}\in Y_{n_{j}}, J_{\lambda}(u_{n_{j}})\rightarrow c,\ \hbox{and}\ \ \  J_{\lambda}|_{Y_{n_{j}}}^{'}(u_{n_{j}})\rightarrow 0,\ \ \  \hbox{as}\ \ n_{j}\rightarrow+\infty,
$$
where $c$ satisfies \eqref{E1}. As in the proof of Lemma \ref{WSo}, we have $u_0$ is a weak solution of problem \eqref{CCE1} with $0<q<1$, where $u_0\in H_{0}^{1}(\Omega)$ is the weak limit of $\{u_{n_{j}}\}$. Taking $\varphi=u_0\in H_{0}^{1}(\Omega)$ as a test function in \eqref{CCE1}, we have
$$
\int_{\Omega}|\nabla u_0|^{2}dx=
\lambda\int_{\Omega}u_0^{q+1}dx+
\int_{\Omega}\int_{\Omega}\frac{|u_0(x)|^{2_{\mu}^{\ast}}|u_0(y)|^{2_{\mu}^{\ast}}}{|x-y|^{\mu}}dxdy,
$$
and so
$$
J_{\lambda}(u_0)=(\frac{\lambda}{2}-\frac{\lambda}{q+1})\int_{\Omega}u_0^{q+1}dx+\frac{N+2-\mu}{4N-2\mu}
\int_{\Omega}\int_{\Omega}\frac{|u_0(x)|^{2_{\mu}^{\ast}}|u_0(y)|^{2_{\mu}^{\ast}}}
{|x-y|^{\mu}}dxdy.
$$
By Hl\"{o}der inequality and $0<q<1$, we have
$$
(\frac{\lambda}{2}-\frac{\lambda}{q+1})\int_{\Omega}u_0^{q+1}dx\geq
(\frac{\lambda}{2}-\frac{\lambda}{q+1})|\Omega|^{\frac{2^{\ast}-q-1}{2^{\ast}}}|u_0|_{2^{\ast}}^{q+1}
\geq-C\lambda \Big(\int_{\Omega}\int_{\Omega}\frac{|u_0(x)|^{2_{\mu}^{\ast}}|u_0(y)|^{2_{\mu}^{\ast}}}
{|x-y|^{\mu}}dxdy\Big)^{\frac{q+1}{2\cdot2_{\mu}^{\ast}}}
$$
and so
\begin{equation}\label{E2}
J_{\lambda}(u_0)\geq\frac{N+2-\mu}{4N-2\mu}
\int_{\Omega}\int_{\Omega}\frac{|u_0(x)|^{2_{\mu}^{\ast}}|u_0(y)|^{2_{\mu}^{\ast}}}
{|x-y|^{\mu}}dxdy-C\lambda \Big(\int_{\Omega}\int_{\Omega}\frac{|u_0(x)|^{2_{\mu}^{\ast}}|u_0(y)|^{2_{\mu}^{\ast}}}
{|x-y|^{\mu}}dxdy\Big)^{\frac{q+1}{2\cdot2_{\mu}^{\ast}}}.
\end{equation}
We define $M_0\geq0$ by
\begin{equation}\label{E3}
\min_{t>0}\Big(\frac{N+2-\mu}{4N-2\mu}t^{2\cdot2_{\mu}^{\ast}}-C\lambda t^{q+1}\Big)=-M_0\lambda^{\frac{2\cdot2_{\mu}^{\ast}}{2\cdot2_{\mu}^{\ast}-q-1}}.
\end{equation}

Denote by  $v_{n_{j}}:=u_{n_{j}}-u_0$, then we know $v_{n_{j}}\rightharpoonup0$ in $H_{0}^{1}(\Omega)$ and $v_{n_{j}}\rightarrow 0$ a.e. in $\Omega$.
From the proof of Lemma \ref{ConPro}, we have
\begin{equation}\label{E4}
c\leftarrow J_{\lambda}(u_{n_{j}})=J_{\lambda}(u_0)+\frac{1}{2}\int_{\Omega}|\nabla v_{n_{j}}|^{2}dx-\frac{1}{2\cdot2_{\mu}^{\ast}}\int_{\Omega}\int_{\Omega}\frac{|v_{n_{j}}(x)|^{2_{\mu}^{\ast}}
|v_{n_{j}}(y)|^{2_{\mu}^{\ast}}}
{|x-y|^{\mu}}dxdy+o_n(1)
\end{equation}
and
\begin{equation}\label{E5}
o_n(1)=\int_{\Omega}|\nabla v_{n_{j}}|^{2}dx-\int_{\Omega}\int_{\Omega}\frac{|v_{n_{j}}(x)|^{2_{\mu}^{\ast}}|v_{n_{j}}(y)|^{2_{\mu}^{\ast}}}
{|x-y|^{\mu}}dxdy.
\end{equation}
From \eqref{E5}, we know there exists a nonnegative constant $b$ such that
$$
\int_{\Omega}|\nabla v_{n_{j}}|^{2}dx\rightarrow b
$$
and
$$
\int_{\Omega}\int_{\Omega}\frac{|v_{n_{j}}(x)|^{2_{\mu}^{\ast}}|v_{n_{j}}(y)|^{2_{\mu}^{\ast}}}
{|x-y|^{\mu}}dxdy\rightarrow b,
$$
as $n_{j}\rightarrow+\infty$.
By the definition of the best constant $S_{H,L}$ in \eqref{S1}, we have
$$
S_{H,L}\Big(\int_{\Omega}\int_{\Omega}\frac{|v_{n_{j}}(x)|^{2_{\mu}^{\ast}}|v_{n_{j}}(y)|^{2_{\mu}^{\ast}}}
{|x-y|^{\mu}}dxdy\Big)^{\frac{N-2}{2N-\mu}}\leq\int_{\Omega}|\nabla v_{n_{j}}|^{2}dx,
$$
which yields $b\geq S_{H,L}b^{\frac{N-2}{2N-\mu}}$. Thus we have either $b=0$ or $b\geq S_{H,L}^{\frac{2N-\mu}{N-\mu+2}}$. If  $b\geq S_{H,L}^{\frac{2N-\mu}{N-\mu+2}}$,  then we obtain from \eqref{E2}, \eqref{E3} and \eqref{E4} that
$$\aligned
c&=\frac{N+2-\mu}{4N-2\mu}b+J_{\lambda}(u_0)\\
&\geq\frac{N+2-\mu}{4N-2\mu}S_{H,L}^{\frac{2N-\mu}{N-\mu+2}}+\frac{N+2-\mu}{4N-2\mu}
\int_{\Omega}\int_{\Omega}\frac{|u_0(x)|^{2_{\mu}^{\ast}}|u_0(y)|^{2_{\mu}^{\ast}}}
{|x-y|^{\mu}}dxdy-C\lambda \Big(\int_{\Omega}\int_{\Omega}\frac{|u_0(x)|^{2_{\mu}^{\ast}}|u_0(y)|^{2_{\mu}^{\ast}}}
{|x-y|^{\mu}}dxdy\Big)^{\frac{q+1}{2\cdot2_{\mu}^{\ast}}}\\
&\geq\frac{N+2-\mu}{4N-2\mu}S_{H,L}^{\frac{2N-\mu}{N-\mu+2}}-
M_0\lambda^{\frac{2\cdot2_{\mu}^{\ast}}{2\cdot2_{\mu}^{\ast}-q-1}},
\endaligned
$$
which contradicts with \eqref{E1}. Thus $b=0$, and
$$
\|u_{n_{j}}-u_0\|\rightarrow0
$$
as $n_{j}\rightarrow+\infty$. This ends the proof of Lemma \ref{MA4}.
\end{proof}

\noindent
{\bf Proof of Theorem \ref{MEXS}.}
Denote by
$$
\beta_{k}:=\sup_{u\in Z_{k},\|u\|=1}|u|_{q+1},
$$
by Lemma 3.8 of \cite{Wi}, we know
\begin{equation}\label{E6}
\beta_{k}\rightarrow0,k\rightarrow\infty.
\end{equation}
From the Sobolev embedding theorem and the Hardy-Littlewood-Sobolev inequality,  exists $R>0$ small enough such that, for any $\|u\|\leq R$,
$$
\frac{1}{2\cdot2_{\mu}^{\ast}}\int_{\Omega}\int_{\Omega}\frac{|u(x)|^{2_{\mu}^{\ast}}|u(y)|^{2_{\mu}^{\ast}}}
{|x-y|^{\mu}}dxdy\leq
\frac{1}{2\cdot2_{\mu}^{\ast}}C_{1}\|u\|^{2(\frac{2N-\mu}{N-2})}\leq\frac{1}{4}\|u\|^{2}.
$$
Thus we get, for all $u\in Z_{k}\backslash\ \{0\}, \|u\|\leq R$,
\begin{equation}\label{E7}
\aligned
J_{\lambda}(u)&\geq\frac{1}{2}\int_{\Omega}|\nabla u|^{2}dx-\frac{\lambda}{q+1}\int_{\Omega}| u|^{q+1}dx-\frac{1}{4}\|u\|^{2}\\
&\geq\frac{1}{4}\|u\|^{2}-\frac{\lambda}{q+1}\beta_{k}^{q+1}\|u\|^{q+1}.\\
\endaligned
\end{equation}
We choose $\rho_{k}:=(\frac{4\lambda\beta_{k}^{q+1}}{q+1})^{\frac{1}{1-q}}$. From \eqref{E6} and $0<q<1$, we have
$\rho_{k}\rightarrow0,k\rightarrow\infty$ and so there exists $k_{0}$ such that for every $\rho_{k}\leq R$ when $k\geq k_{0}$. Thus, for every $k\geq k_{0}$, there exists $\rho_{k}>0$ such that $a_{k}=\inf_{u\in Z_{k},\|u\|=\rho_{k}}J_{\lambda}(u)\geq0$. Condition $(B_{1})$ is thus proved.

Since $Y_{k}$ is a finite dimensional subspace of $H_{0}^{1}(\Omega)$, we have all norms on $Y_{k}$ are equivalent and so condition $(B_{2})$ is satisfied for every $r_{k}>0$ small enough when $\lambda>0$.

By \eqref{E7}, we know for all $u\in Z_{k}\backslash\ \{0\}, \|u\|\leq \rho_{k}, k\geq k_{0}$,
$$
J_{\lambda}(u)\geq-\frac{\lambda}{q+1}\beta_{k}^{q+1}\|u\|^{q+1}
\geq-\frac{\lambda}{q+1}\beta_{k}^{q+1}\rho_{k}^{q+1}.
$$
Then condition $(B_{3})$ is satisfied from $\beta_{k}\rightarrow0$, $\rho_{k}\rightarrow0$, $k\rightarrow\infty$.

We know that there exists $\lambda^{\ast}>0$ such that, for every $0<\lambda<\lambda^{\ast}$ and $c<0$, $J_{\lambda}$ satisfies the $(PS)_c^{\ast}$ condition from Lemma \ref{MA4}. By Theorem \ref{MA3}, we have there exists $\lambda^{\ast}>0$ such that, for every $0<\lambda<\lambda^{\ast}$, problem \eqref{CCE1} has a sequence of solutions $\{u_{n}\}\subset H_{0}^{1}(\Omega)$ such that $J_{\lambda}(u_{n})\rightarrow0$, $n\rightarrow\infty$. $\hfill{} \Box$

To obtain the multiplicity results for the subcritical nonlocal case, we need to recall the famous Fountain Theorem in \cite{B} which states as
\begin{thm}\label{MA5}
Let $X$ be a Banach space, $I\in C^1(X,\R)$ is an even functional. If, for every $k\in\N$, there exists $\rho_{k}>r_{k}>0$ such that \\
$(A_{1})$ $a_{k}:=\max_{u\in Y_{k},\|u\|=\rho_{k}}I(u)\leq0$,\\
$(A_{2})$ $b_{k}:=\inf_{u\in Z_{k},\|u\|=r_{k}}I(u)\rightarrow\infty$, $k\rightarrow\infty$,\\
$(A_{3})$ $I$ satisfies the $(PS)_c$ condition for every $c>0$,\\
then $I$ has an unbounded sequence of critical values converging to $\infty$.
\end{thm}

\noindent
\textbf{Proof of Theorem \ref{MEXS2}}.
We denote the energy functional associated to equation \eqref{SCCE} by
$$
J_{\lambda,p}(u)=\frac{1}{2}\int_{\Omega}|\nabla u|^{2}dx-\frac{1}{2p}\int_{\Omega}
\int_{\Omega}\frac{|u(x)|^{p}|u(y)|^{p}}
{|x-y|^{\mu}}dxdy-\frac{\lambda}{2}\int_{\Omega}|u|^{q+1}dx.
$$
It is standard to prove the $ (PS) $ condition $(A_{3})$ holds everywhere and we only need to check the conditions $(A_{1})$ and $(A_{2})$ hold. In fact,
similar to the proof Lemma 2.3 in \cite{GY}, we have
$$
\|\cdot\|_{2}:=\Big(\int_{\Omega}\int_{\Omega}\frac{|\cdot|^{p}|\cdot|^{p}}
{|x-y|^{\mu}}dxdy\Big)^{\frac{1}{2p}}
$$
defines a norm on $L^{\frac{2Np}{2N-\mu}}(\Omega)$. Thus
$$
J_{\lambda,p}(u)\leq\frac{1}{2}\int_{\Omega}|\nabla u|^{2}dx-\frac{1}{2p}\|u\|_{2}^{2p}.
$$
Since $Y_{k}$ is a finite dimensional subspace of $H_{0}^{1}(\Omega)$, we have all norms on $Y_{k}$ are equivalent and so relation $(A_{1})$ is satisfied for every $\rho_{k}>0$ large enough when $\lambda>0$.

We denote
$$
\beta_{k}:=\sup_{u\in Z_{k},\|u\|=1}|u|_{\frac{2Np}{2N-\mu}}
$$
and have
\begin{equation}\label{E9}
\beta_{k}\rightarrow0,k\rightarrow\infty
\end{equation}
by Lemma 3.8 of \cite{Wi}. By the Hardy-Littlewood-Sobolev inequality, we obtain
$$
\int_{\Omega}\int_{\Omega}\frac{|u(x)|^{p}|u(y)|^{p}}
{|x-y|^{\mu}}dxdy\leq C_{1}|u|_{\frac{2Np}{2N-\mu}}^{2p}.
$$
Thus we get, for all $u\in Z_{k}\backslash\ \{0\}$,
\begin{equation}\label{E10}
\aligned
J_{\lambda}(u)&\geq\frac{1}{2}\int_{\Omega}|\nabla u|^{2}dx-\frac{\lambda}{q+1}\int_{\Omega}| u|^{q+1}dx-C_{2}|u|_{\frac{2Np}{2N-\mu}}^{2p}\\
&\geq\frac{1}{2}\|u\|^{2}-C_{2}\beta_{k}^{2p}\|u\|^{2p}-\frac{\lambda}{q+1}\int_{\Omega}| u|^{q+1}dx.\\
\endaligned
\end{equation}
We choose $r_{k}:=(2pC_{2}\beta_{k}^{2p})^{\frac{1}{2-2p}}$ and get, for all $u\in Z_{k}\backslash\ \{0\}$ and $ \|u\|= r_{k}$,
$$
J_{\lambda}(u)\geq(\frac{1}{2}-\frac{1}{2p})(2pC_{2}\beta_{k}^{2p})^{\frac{1}{2-2p}}
-\frac{\lambda}{q+1}\int_{\Omega}| u|^{q+1}dx.
$$
From \eqref{E9}, we have that relation $(A_{2})$ is proved.  $\hfill{} \Box$

\vspace{2cm}
\noindent {\bf Acknowledgements}
The authors would like to thank  the anonymous referee
for his/her useful comments and suggestions which help to improve the presentation of the paper greatly.

\end{document}